\documentclass[12pt]{amsart}
\usepackage{amsmath, amsthm, amssymb, url}
\usepackage[all]{xy}

\input xy
\xyoption{all}

\newtheorem{thm}{Theorem}[section]
\newtheorem{lem}[thm]{Lemma}
\newtheorem{cor}[thm]{Corollary}

\newtheorem{prop}[thm]{Proposition}

\theoremstyle{definition}

\newtheorem{rmk}[thm]{Remark}

\newenvironment{axiom}{\begin{list}{$\bullet$}{\setlength{\labelsep}{.7cm}
\setlength{\leftmargin}{2.5cm}\setlength{\rightmargin}{0cm}%
\setlength{\labelwidth}{1.8cm}\setlength{\itemsep}{0pt}}}{\end{list}}
\newcommand{\ax}[1]{\item[{\bf #1}\hfill]\index{#1}}

\def\bF{\mathbb{F}}
\def\bN{\mathbb{N}}
\def\bQ{\mathbb{Q}}
\def\bZ{\mathbb{Z}}

\def\cB{\mathcal{B}}

\def\cL{\mathcal{L}}
\def\cT{\mathcal{T}}
\def\cM{\mathcal{M}}

\def\cO{\mathcal{O}}

\def\rQ{\mathrm{Q}}

\def\alg{\operatorname{alg}}
\def\dim{\operatorname{dim}}

\def\trdeg{\operatorname{trdeg}}

\def\sep{\operatorname{sep}}
\def\rac{\operatorname{rac}}
\def\VF{\operatorname{VF}}

\newcommand{\ec}{\prec_{\exists}}
\newcommand{\bfind}[1]{\index{#1}{\bf #1}}

\font\teneu=eufm10 scaled 1200
\font\seveneu=eufm7 scaled 1200
\font\fiveeu=eufm5 scaled 1200

\def\eu #1{{\mathchoice{{\hbox{\teneu #1}}}{{\hbox{\teneu #1}}}
{{\hbox{\seveneu #1}}}{{\hbox{\fiveeu #1}}}}}

\begin{document}

\title{The model theory of separably tame valued fields}
\author{Franz--Viktor Kuhlmann}
\address{Department of Mathematics and Statistics,
University of Saskatchewan, 106 Wiggins Road,
Saskatoon, Saskatchewan, Canada S7N 5E6}
\email{fvk@math.usask.ca}
\thanks{The first-named author is partially supported by a Canadian NSERC grant.}
\author{Koushik Pal}
\address{Department of Mathematics and Statistics,
University of Saskatchewan, 106 Wiggins Road,
Saskatoon, Saskatchewan, Canada S7N 5E6}
\email{koushik.pal@usask.ca}
\thanks{The second-named author is partially supported by a PIMS Postdoctoral Fellowship.}
\maketitle

\begin{abstract}\noindent
A henselian valued field $K$ is called separably tame if its separable-algebraic closure $K^{\sep}$ is a tame extension, that is, the ramification field of the normal extension $K^{\sep}|K$ is separable-algebraically closed. Every separable-algebraically maximal Kaplansky field is a separably tame field, but not conversely. In this paper, we prove Ax--Kochen--Ershov Principles for separably tame fields. This leads to model completeness and completeness results relative to the value group and residue field. As the maximal immediate extensions of separably tame fields are in general not unique, the proofs have to use much deeper valuation theoretical results than those for other classes of valued fields which have already been shown to satisfy Ax--Kochen--Ershov Principles. Our approach also yields alternate proofs of known results for separably closed valued fields.
\end{abstract}

\section{Introduction}
In this paper, we consider valued fields. By $(K,v)$ we mean a field $K$ equipped with a valuation $v$. We write a valuation in the classical additive (Krull) way, that is, the value group is an additively written ordered abelian group, the homomorphism property of $v$ says that $vab= va+vb$, and the ultrametric triangle law says that $v(a+b)\geq\min \{va,vb\}$. Further, we have the rule $va=\infty \Leftrightarrow a=0$. We denote the value group by $vK$, the residue field by $Kv$ and the valuation ring by $\cO_v$ or $\cO_K$. For elements $a\in K$, the value is denoted by $va$, and the residue by $av$. By a valued field extension $(L|K, v)$ we mean that $(L, v)$ is a valued field, $L|K$ is a field extension, and $K$ is endowed with the restriction of $v$.

Our main concern is the model theory of separably tame valued fields, which we will introduce now.

A valued field is {\bf henselian} if it satisfies Hensel's Lemma, or equivalently, if it admits a unique extension of the valuation to every algebraic extension field. The {\bf henselization} of a valued field $(L, v)$, denoted by $(L, v)^h$ or simply $L^h$, is the ``minimal'' extension of $(L, v)$ which is henselian. It is unique up to isomorphism of valued fields. The henselization is an immediate separable-algebraic extension.

Every finite extension $(E|K, v)$ of valued fields satisfies the {\bf fundamental inequality}:
\begin{equation}		\label{fundineq}
n\>\geq\>\sum_{i=1}^{\rm g} {\rm e}_i {\rm f}_i
\end{equation}
where $n = [E : K]$ is the degree of the extension, $v_1,\ldots,v_{\rm g}$ are the distinct extensions of $v$ from $K$ to $E$, ${\rm e}_i = (v_iE : vK)$ are the respective ramification indices and ${\rm f}_i = [Ev_i : Kv]$ are the respective inertia degrees. The extension is called {\bf defectless} if equality holds in (\ref{fundineq}). A valued field $(K, v)$ is called {\bf defectless} (or {\bf stable}) if each of its finite extensions is defectless, and {\bf separably defectless} if each of its finite separable extensions is defectless. If char $Kv = 0$, then $(K, v)$ is defectless (this is a consequence of the ``Lemma of Ostrowski'' \cite[Section 2.2]{FVK_TF}). Note that ${\rm g} = 1$ if $(K, v)$ is henselian, in which case (\ref{fundineq}) becomes $n \ge {\rm e\,f}$. In particular, if a valued field $(K, v)$ is henselian and defectless, then we have $n = {\rm e\,f}$, i.e., $[E : K] = [vE : vK][Ev : Kv]$ for all finite extensions $(E|K, v)$.

Take a henselian field $(K, v)$, and let $p$ denote the {\bf characteristic exponent} of its residue field $Kv$, i.e., $p=\mbox{char }Kv$ if this is positive, and $p=1$ otherwise. An algebraic extension $(L|K, v)$ of a henselian field $(K, v)$ is called {\bf tame} if every finite subextension $E|K$ of $L|K$ satisfies the following conditions:
\begin{axiom}
\ax{(TE1)} The ramification index $(vE : vK)$ is prime to $p$,
\ax{(TE2)} The residue field extension $Ev|Kv$ is separable,
\ax{(TE3)} The extension $(E|K, v)$ is defectless.
\end{axiom}

A {\bf tame valued field} (in short, {\bf tame field}) is a henselian field for which all algebraic extensions are tame. Equivalently, a tame field is one whose algebraic closure is equal to the ramification field $K^r$ of the normal extension $K^{\sep}|K$, where $K^{\sep}$ denotes the separable-algebraic closure of $K$ \cite[Lemma 2.17(a)]{FVK_TF}. Likewise, a {\bf separably tame field} is a henselian field for which all separable-algebraic extensions are tame. Equivalently, a separably tame field is one whose separable-algebraic closure is equal to $K^r$ \cite[Lemma 2.17(a)]{FVK_TF}. The algebraic properties of tame fields and separably tame fields have been studied in \cite{FVK_TF}, and some of those properties will be mentioned in Section~\ref{sectastf}.

One useful thing to note here is that a separably tame field is trivially a tame field when its characteristic is 0, because then every algebraic extension $L|K$ is separable. Thus, in the residue characteristic 0 case and in the mixed characteristic case, a separably tame field is a tame field. That is why, when studying the model theory of separably tame fields, we restrict our attention without loss of generality to the (positive equi-characteristic) case when char $K$ = char $Kv = p$, where $p$ is a prime.

An extension $(L|K,v)$ of valued fields is called {\bf immediate} if the canonical embeddings $vK\hookrightarrow vL$ and $Kv\hookrightarrow Lv$ are onto. A valued field is called {\bf algebraically maximal} if it does not admit proper immediate algebraic extensions; it is called {\bf separable-algebraically maximal} if it does not admit proper immediate separable-algebraic extensions. Every separable-algebraically maximal valued field is, therefore, a henselian field. Also, by Zorn's Lemma, every valued field admits a maximal (algebraic, separable-algebraic, or respectively transcendental) immediate extension.

Take a valued field $(K, v)$ and denote the characteristic exponent of $Kv$ by $p$. Then $(K, v)$ is a {\bf Kaplansky field} if $vK$ is $p$-divisible and $Kv$ does not admit any finite extension whose degree is divisible by $p$. All algebraically maximal Kaplansky fields are tame fields \cite[Corollary 3.3]{FVK_TF}. But the converse does not hold since for a tame field it is admissible that its residue field has finite separable extensions with degree divisible by $p$. Similarly, all separable-algebraically maximal Kaplansky fields are separably tame fields \cite[Corollary 3.11(a)]{FVK_TF}, but not conversely. It is because of the latter fact that the uniqueness of maximal immediate extensions will in general fail (cf.~\cite{FVKMPPR}). This is what makes the proof of the model theoretic results for tame and separably tame fields much harder than for algebraically and separable-algebraically maximal Kaplansky fields.

Let us now give a quick introduction of the basic notions in the model theory of valued fields that will be used in the questions we will ask for separably tame fields.

We take $\cL_{\VF} := \{+, -, \cdot\,, \mbox{}^{-1}, 0, 1, \cO\}$ to be the language of valued fields, where $\cO$ is a binary relation symbol for valuation divisibility. That is, $\cO(a,b)$ will be interpreted by $va\geq vb$, or equivalently, $a/b$ being an element of the valuation ring $\cO_v\,$. We will write $\cO(X)$ in place of $\cO(X,1)$ (note that $\cO(a,1)$ says that $va\geq v1 = 0$, i.e., $a\in\cO_v$).

For $(K,v)$ and $(L,v)$ to be elementarily equivalent in the language of valued fields, it is necessary that $vK$ and $vL$ are elementarily equivalent in the language $\cL_\mathrm{OG} $ \linebreak $ := \{+, -, 0, <\}$ of ordered groups, and that $Kv$ and $Lv$ are elementarily equivalent in the language $\cL_\mathrm{F} := \{+, -, \cdot\,, \mbox{ }^{-1}, 0, 1\}$ of fields (or in the language $\cL_\mathrm{R} := \{+, -, \cdot\,, 0, 1\}$ of rings). This is because elementary sentences about the value group and about the residue field can be encoded in the valued field itself.

Our main concern in this paper is to find additional conditions on $(K,v)$ and $(L,v)$ under which these necessary conditions are also sufficient, i.e., the following {\bf Ax--Kochen--Ershov Principle} (in short: AKE$^\equiv$ Principle) holds:
\begin{equation}		\label{AKEequiv}
vK\equiv vL\>\wedge\> Kv\equiv Lv \;\;\;\Longrightarrow\;\;\; (K,v)\equiv \mbox{$(L,v)$}\;.
\end{equation}

An AKE$^\prec$ Principle is the following analogue for elementary extensions:
\begin{equation}		\label{AKEprec}
(K,v)\subseteq (L,v)\>\wedge\> vK\prec vL\>\wedge\> Kv\prec Lv\;\;\;\Longrightarrow\;\;\;(K,v)\prec \mbox{$(L,v)$}\;.
\end{equation}

If $\cM$ is an $\cL$-structure and $\cM'$ a substructure of $\cM$, then we say that $\cM'$ is {\bf existentially closed} in $\cM$, and write $\cM'\ec\cM$, if every existential $\cL$-sentence with parameters from $\cM'$ which holds in $\cM$ also holds in $\cM'$. The corresponding AKE$^\exists$ Principle is then:
\begin{equation}		\label{AKE}
(K,v)\subseteq (L,v)\>\wedge\> vK\ec vL\>\wedge\> Kv\ec Lv \;\;\;\Longrightarrow\;\;\; (K,v)\ec\mbox{$(L,v)$}\;.
\end{equation}
The conditions
\begin{equation}		\label{sico}
vK\ec vL\mbox{\ \ and\ \ } Kv\ec Lv
\end{equation}
will be called {\bf side conditions}. It is an easy exercise in model theoretic algebra to show that these conditions imply that $vL/vK$ is torsion free and that $Lv|Kv$ is \bfind{regular}, i.e., the algebraic closure of $Kv$ is linearly disjoint from $Lv$ over $Kv$, or equivalently, $Kv$ is relatively algebraically closed in $Lv$ and $Lv|Kv$ is separable; cf.\ Lemma~\ref{sica}.

A valued field for which (\ref{AKE}) holds will be called an {\bf AKE$^\exists$-field}. A class {\bf C} of valued fields will be called an {\bf AKE$^\equiv$-class} (respectively, {\bf AKE$^\prec$-class}) if (\ref{AKEequiv}) (respectively, (\ref{AKEprec})) holds for all $(K, v), (L, v)\in$ {\bf C}, and it will be called an {\bf AKE$^\exists$-class} if (\ref{AKE}) holds for all $(K, v)\in$ {\bf C}. We will also say that {\bf C} is {\bf relatively complete} if it is an AKE$^\equiv$-class, and that {\bf C} is {\bf relatively model complete} if it is an AKE$^\prec$-class. Here, ``relatively'' means ``relative to the value groups and residue fields''.

Things are easy when $(K, v)$ is trivially valued. From the fundamental inequality one deduces that $vL/vK$ is torsion for every algebraic extension $L|K$. Since $vK = \{0\}$, it follows that $vL = \{0\}$. Thus, $K\cong Kv$, and this isomorphism extends to an isomorphism of $L$ and $Lv$. This yields that $[L:K] = [Lv:Kv]$ for every finite extension $L|K$ showing that the extension is defectless, and that there are no proper immediate extensions. This in turn implies that $(K, v)$ is henselian, and that any separable and finite extension $(L|K, v)$ is a tame extension. In particular, any trivially valued field is separably tame. Analogously, if $(L|K, v)$ is a valued field extension satisfying the side conditions, then $(L, v)$ is again trivially valued. Consequently, $K\cong Kv$ and $L\cong Lv$. In particular, $(L|K, v)$ satisfies all three AKE Principles.


Almost immediately one also notices that purely inseparable extensions are a serious obstruction to obtaining such AKE Principles, even for separable-algebraically closed valued fields, let alone separably tame fields. For example, take a separable-algebraically closed nontrivially valued field $(K, v)$ of characteristic $p > 0$ such that its algebraic closure $K^{\alg}$ is a proper extension. Clearly $(K^{\alg}, v)$ is also separable-algebraically closed. Since the residue field of $(K, v)$ is algebraically closed and the value group is divisible (see Lemma~\ref{sep_perf_hull}), it follows that $(K^{\alg}|K, v)$ is an immediate extension. In particular, the side conditions hold. Now take $a\in K^{\alg}\setminus K$. Since $K$ is separable-algebraically closed, the minimal polynomial $g(X)$ of $a$ over $K$ is a purely inseparable polynomial. But then $(K^{\alg}, v)\models\exists x\;(g(x) = 0)$, while $(K, v)\models\neg\exists x\;(g(x) = 0)$. Thus, even the AKE$^\exists$ Principle fails for separable-algebraically closed, and hence separably tame, valued fields in $\cL_{\VF}$.

So, at the very least we can expect AKE Principles only for separable extensions of separably tame fields. To that end, we consider a different language $\cL_\rQ$ for valued fields such that, considered as $\cL_\rQ$-structures, one field is an extension of another provided it is a
separable extension. We achieve this by adding to $\cL_{\VF}$ predicates $(Q_m(x_1, \ldots, x_m))_{m\in\omega}$, which model the notion of $p$-independence. The details are given in Section~\ref{modprelim}. This language was first used by Ershov \cite{YLE_LQ} and Wood \cite{CW_SCF} in the context of separable-algebraically closed fields, and later by Delon \cite{FD_thesis} in the context of separable-algebraically maximal Kaplansky fields. 

Another thing to consider in obtaining the AKE Principles for separably tame fields in general is the notion of {\em $p$-degree}, also known as the {\em Ershov invariant}. The definition and details are given in Section~\ref{algprelim}. The $p$-degree of a valued field of characteristic $p$ can be either finite or $\infty$. It is an isomorphism-invariant for valued fields. Being an elementary property in the language $\cL_\rQ$, it features as an essential property in the elementary theory of separably tame fields. The case of infinite $p$-degree is hard to deal with. In this paper, we therefore restrict our attention to the case of finite $p$-degrees. 

The hardest problem, of course, is caused by the non-uniqueness of maximal immediate extensions of separably tame fields. For a list of elementary classes of valued fields that are known to satisfy all or some of the AKE Principles, refer to the Introduction in \cite{FVK_TF}. All the valued fields mentioned in that list have the common property that their maximal (algebraic, separable-algebraic, or transcendental) immediate extensions are unique up to valuation preserving isomorphism. However, as is shown in \cite{FVK_TF} and as we will show again in this paper, this uniqueness is not indispensable.

We will now list the main theorems of this paper. In order to do that, we need one more definition that will be fundamental for this paper.

Let $(L|K, v)$ be any extension of valued fields. Assume that $L|K$ has finite transcendence degree. Then (by \cite[Corollary 2.3]{FVK_TF}),
\begin{equation}		\label{wtdgeq}
\trdeg L|K \>\geq\> \trdeg Lv|Kv \,+\, \dim_\bQ \bQ\otimes vL/vK\;.
\end{equation}
The last term in (\ref{wtdgeq}) is called the {\bf rational rank} ({\bf r.r.}) of $vL$ over $vK$. We say that $(L|K, v)$ is {\bf without transcendence defect} if equality holds in (\ref{wtdgeq}). If $L|K$ does not have finite transcendence degree, then we say that $(L|K, v)$ is without transcendence defect if every subextension of finite transcendence degree is. In Section~\ref{sectmwtd} we show that

\begin{thm}		\label{sewtd}
Every separable extension without transcendence defect of a separably tame field satisfies the AKE$^\exists$ Principle in $\cL_{\VF}$.
\end{thm}

A more general version of this result is already proved in \cite{FVK_TF} (see Theorem~\ref{stake}). However, we reprove the above special case because we give an alternate proof --- one that constructs appropriate embeddings of valued fields while respecting certain embeddings at the level of value groups and residue fields given a priori (see Section~\ref{sectmwtd}). This is crucial for obtaining the other AKE Principles in the context of separably tame fields.

\par\smallskip
The following is the main theorem of this paper.

\begin{thm}		\label{tameAKE}
The class of all separably tame fields of a fixed characteristic $p$ and a fixed finite $p$-degree $e$ is an AKE$^\exists$-class and an AKE$^\prec$-class in $\cL_\rQ$, and an AKE$^\equiv$-class in $\cL_{\VF}$.
\end{thm}

As an immediate consequence of the foregoing theorem, we get the following criterion for decidability.
\begin{thm}		\label{dec}
Let $(K, v)$ be a separably tame field of positive characteristic $p$ and finite $p$-degree $e$. Assume that the elementary theories $\mathrm{Th}(vK)$ of its value group (as an ordered abelian group) and $\mathrm{Th}(Kv)$ of its residue field (as a field) both admit recursive elementary axiomatizations. Then also the elementary theory of $(K, v)$ as a valued field admits a recursive elementary axiomatization and is decidable in $\cL_{\VF}$.
\end{thm}
Indeed, the axiomatization of $\mbox{\rm Th}(K, v)$ can be taken to consist of the axioms of separably tame fields of characteristic $p > 0$ and finite $p$-degree $e$, together with the translations of the axioms of $\mbox{\rm Th}(vK)$ and $\mbox{\rm Th}(Kv)$ to the language of valued fields (cf.\ Lemma~\ref{elprvgrf}).


\par\smallskip
\par\smallskip
In the last section, we give an application of our main results to the cases of separable-algebraically closed valued fields and separable-algebraically maximal Kaplansky fields. Our approach gives alternate proofs to well-known results in this context.

\par\smallskip
\par\smallskip
We deduce our model theoretic results for separably tame fields from two main theorems. The first theorem is an improved version of \cite[Proposition 2.3]{HKFVK} for separable function field extensions, which answers the concern raised in \cite[Remark 2.4]{HKFVK}. We will prove it in Section~\ref{strvff}.


\begin{thm}		\label{sepdefextn}
Let $(F|K, v)$ be a separable valued function field without transcendence defect with r.r. $vF/vK = r$ and $\trdeg Fv|Kv = s$. Then there is a transcendence basis $\cT = \{x_1, \ldots, x_r, y_1, \ldots, y_s\}$ of $F|K$ such that $F|K(\cT)$ is separable, with
\begin{eqnarray*}
vK(\cT) & = & vK \oplus \bigoplus_{1\le i \le r} \bZ vx_i, \;\;\;\;\mbox{and}\\
K(\cT)v & = & Kv(y_1v, \ldots, y_sv).
\end{eqnarray*}
If, in addition, $vF/vK$ is torsion free and $Fv|Kv$ is separable, then we can choose $\cT$ which additionally satisfies that $vK(\cT) = vF$ and $Fv|K(\cT)v$ is separable.
\end{thm}

The second fundamental theorem, originally proved in \cite{FVK_thesis}, is a structure theorem for separable immediate valued function fields over separably tame fields.

\begin{thm}		\label{stt3}
Take an immediate function field $(F|K, v)$ of transcendence degree 1. Assume that $(K, v)$ is a separably tame field and $F|K$ is separable. Then
\begin{equation}
\mbox{there is $x\in F$ such that }\; (F^h, v)\,=\,(K(x)^h, v)\;.
\end{equation}
\end{thm}

\section{Preliminaries}
\subsection{Valuation Theoretical Preliminaries}


We will denote the algebraic closure of a field $K$ by $K^{\alg}$. Whenever we have a valuation $v$ on $K$, we will automatically fix an extension of $v$ to the algebraic closure $K^{\alg}$ of $K$. It does not play a role which extension we choose, except if $v$ is also given on an extension field $L$ of $K$; in this case, we choose the extension to $K^{\alg}$ to be the restriction of the extension to $L^{\alg}$. We say that $v$ is {\bf trivial} on $K$ if $vK = \{0\}$; otherwise $(K, v)$ is said to be {\bf nontrivially valued}. If the valuation $v$ of $L$ is trivial on the subfield $K$, then we may assume that $K$ is a subfield of $Lv$ and the residue map $K\ni a\mapsto av$ is the identity.

We will denote by $K^{\sep}$ the separable-algebraic closure of $K$, and by $K^{1/p^{\infty}}$ its perfect hull. If $\Gamma$ is an ordered abelian group and $p$ a prime, then we write $\frac{1}{p^{\infty}}\Gamma$ for the $p$-divisible hull of $\Gamma$, endowed with the unique extension of the ordering from $\Gamma$. We leave the easy proof of the following lemma to the reader.

\begin{lem}		\label{sep_perf_hull}
If $K$ is an arbitrary field and $v$ is a valuation on $K^{\sep}$, then $vK^{\sep}$ is the divisible hull of $vK$, and $(Kv)^{\sep}\subseteq K^{\sep}v$. If, in addition, $v$ is nontrivial on $K$, then $K^{\sep}v$ is the algebraic closure of $Kv$.

Every valuation $v$ on $K$ has a unique extension to $K^{1/p^\infty}$, and it satisfies $vK^{1/p^\infty} = \dfrac{1}{p^{\infty}}vK$ and $K^{1/p^\infty}v = (Kv)^{1/p^\infty}$.
\end{lem}

If $(L|K, v)$ is an extension of valued fields, then a transcendence basis $\cT$ of $L|K$ will be called a {\bf standard valuation transcendence basis} of $(L, v)$ over $(K, v)$ if $\cT = \{x_i, y_j \mid i\in I, j\in J\}$ where the values $\{vx_i\mid i\in I\}$ form a maximal set of values in $vL$ rationally independent over $vK$, and the residues $\{y_jv\mid j\in J\}$ form a transcendence basis of $Lv|Kv$. We then have

\begin{cor}\cite[Corollary 2.4]{FVK_TF}			\label{svtb->wtd}
If a valued field extension admits a standard valuation transcendence basis, then it is an extension without transcendence defect.
\end{cor}

As mentioned in the Introduction, a defectless (respectively, separably defectless) field is a valued field for which all finite (respectively, all finite and separable) extensions are defectless, i.e., satisfy equality in the fundamental inequality (\ref{fundineq}). We will be using the following properties of defectless and separably defectless fields in this paper. The details can be found in \cite{FVK_ER}, \cite{FVK_TF}.
\begin{lem}		\label{defless}
The following are some properties of defectless and separably defectless fields.
\begin{itemize}
\item Every trivially valued field $(K, v)$ is defectless.
\item Every valued field $(K, v)$ with char $Kv = 0$ is a defectless field.
\item Every finite algebraic extension of a defectless field is again a defectless field.
\item Every finite separable-algebraic extension of a separably defectless field is again separably defectless.
\end{itemize}
\end{lem}

As mentioned in the Introduction, all algebraically maximal and separable-algebraically maximal fields are henselian because the henselization is an immediate separable-algebraic extension and therefore these fields must coincide with their henselization. Every henselian defectless (respectively, separably defectless) field is algebraically maximal (respectively, separable-algebraically maximal). However, the converse is not true in general: algebraically maximal (respectively, separable-algebraically maximal) fields need not be defectless (respectively, separably defectless). An additional condition is needed, see \cite{FVK_AS}. Note that for a valued field of residue characteristic 0, ``henselian'', ``algebraically maximal'', ``separable-algebraically maximal'', ``henselian defectless'' and ``henselian separably defectless'' are all equivalent. We will need the following characterization of separable-algebraically maximal fields; cf.\ Theorems 1.6 and 1.8 of \cite{FVK_AS}.
\begin{thm}		\label{exmaxvfth}
A valued field $(K,v)$ is separable-algebraically maximal if and only if it is henselian and for every separable polynomial $f\in K[X]$,
\begin{equation*}		\label{exmaxvf}
\exists x\in K\, \forall y\in K\>\Big(vf(x)\geq vf(y)\Big).
\end{equation*}
\end{thm}

We will assume the reader to be familiar with the theory of pseudo-Cauchy sequences as developed in \cite{IK}. Recall that a pseudo-Cauchy sequence $\{x_\alpha\}_{\alpha < \lambda}$ (with $\lambda$ a limit ordinal) in $(K, v)$ is of {\bf transcendental type} if it fixes the value of every polynomial $f \in K[X]$, that is, $\{vf(x_\alpha)\}$ is constant for all large enough $\alpha < \lambda$. And $\{x_\alpha\}_{\alpha < \lambda}$ is of {\bf algebraic type} if there is a polynomial $f\in K[X]$ and $\alpha_0 < \lambda$ such that $\{vf(x_\alpha)\}_{\alpha_0\le\alpha < \lambda}$ is strictly increasing. In the later case, there is a polynomial $f\in K[X]$ of minimal degree with the said property. Such a polynomial is called a {\bf minimal polynomial} of $\{x_\alpha\}_{\alpha < \lambda}$.

\subsection{Algebraic Preliminaries}		\label{algprelim}

We will use the notion of a $p$-basis in the proof of Theorem~\ref{sepdefextn}. So, let us first define the notions of ``$p$-degree'' and ``$p$-basis'' and describe their relations to (not necessarily algebraic) separable extensions. For further details, consult \cite{MFMJ}, \cite{NJ}, \cite{SL_AG}, \cite{SL}.

Let $K\subseteq L$ be fields of characteristic $p > 0$. We say that $L$ is a {\bf separable extension} of $K$ if the extensions $L|K$ and $K^{1/p}|K$ are linearly disjoint. If $L$ is algebraic over $K$, this is equivalent to $L\subseteq K^{\sep}$. In characteristic zero, every field extension is separable. Some of the algebraic properties satisfied by separable extensions are as follows. In this paper, we denote the field compositum of two fields $L$ and $E$ by $LE$.
\begin{thm}		\label{sepextns}
Separable extensions have the following algebraic properties:
\begin{enumerate}
\item If $K$ is perfect, then every extension $L|K$ is separable.
\item If $E|K$ and $L|E$ are separable extensions, then so is $L|K$.
\item If $L|K$ is separable and $E|K$ is a subextension of $L|K$, then $E|K$ is also separable.
\item If $L|K$ and $E|K$ are linearly disjoint, then $L|K$ is separable if and only if $LE|E$ is separable.
\item Every finitely generated subextension $F|K$ of a separable extension $L|K$ is {\bf separably generated}, i.e., there is a transcendence basis $\{t_1, \ldots, t_n\}$ of $F|K$ such that $F|K(t_1, \ldots, t_n)$ is separable-algebraic. Such a transcendence basis is called a {\bf separating transcendence basis}. Moreover, such a transcendence basis can be chosen from any set of generators.
\end{enumerate}
\end{thm}

Again let $K$ be a field of characteristic $p > 0$. Then the map $\eta: K\longrightarrow K^p: x\mapsto x^p$ is a field endomorphism. In particular, $K^p$ is a subfield of $K$, and thus $K$ is naturally a $K^p$-vector space. We say that $\cB\subseteq K$ is {\bf $p$-independent} in $K$ if the set $\cM$ of all monomials in $\cB$ of the form $b^{i_1}_1 \cdots b^{i_n}_n$, with $b_1, \ldots, b_n\in\cB$ and $0\le i_1, \ldots, i_n\le p - 1$, is linearly independent in the $K^p$-vector space $K$. If, furthermore, $\cM$ is a basis of the $K^p$-vector space $K$, then we call $\cB$ a {\bf $p$-basis} of $K$. $\cB$ is a $p$-basis of $K$ if and only if $\cB$ is a maximal $p$-independent subset of $K$. In particular,
\begin{enumerate}
\item Any $\cB\subseteq K$ is $p$-independent in $K$ if and only if for every $b\in\cB,\; b\not\in K^p[\cB\setminus\{b\}]$.
\item Any $p$-independent subset $\cB\subseteq K$ is a $p$-basis of $K$ if and only if $K^p[\cB] = K$.
\end{enumerate}
Any $p$-independent subset of $K$ extends to a $p$-basis of $K$, and any two $p$-bases of $K$ have the same cardinality, which by convention is an element of $\bN\cup\{\infty\}$. If $|\cB| = e\in\bN$, then $[K : K^p] = p^e$. The size of a $p$-basis $\cB$ of $K$ is called the {\bf $p$-degree} or the {\bf degree of imperfection} or the {\bf Ershov invariant} of $K$. It remains invariant under field isomorphism.

Analogously, if $E$ is a subfield of $K$, we say that $\cB$ is a {\bf $p$-basis of $K$ over $E$} if the set $M$ of all monomials in $\cB$ of the form $b^{i_1}_1\cdots b^{i_n}_n$ with $b_1, \ldots, b_n\in\cB$ and $0\le i_1, \ldots, i_n\le p - 1$, is a basis of the $EK^p$-vector space $K$. Then $K = EK^p[\cB]$.
\begin{thm}		\label{pindep}
The following conditions are equivalent:
\begin{enumerate}
\item[(1)] $L|K$ is separable.
\item[(2)] $L^p$ and $K$ are linearly disjoint over $K^p$.
\item[(3)] $L$ and $K^{1/p^\infty}$ are linearly disjoint over $K$.
\item[(4)] Any $p$-basis of $K$ remains $p$-independent in $L$.
\item[(4')] Some $p$-basis of $K$ remains $p$-independent in $L$.
\item[(5)] Any $p$-basis of $K$ is contained in a $p$-basis of $L$.
\item[(5')] Some $p$-basis of $K$ is contained in a $p$-basis of $L$.
\end{enumerate}
\end{thm}

We will also need the following two results about separable extensions.

\begin{lem}		\label{relsep}
If $L|K$ is a separable field extension and $K_1$ is an algebraic extension of $K$ in $L$, then $L|K_1$ is also separable. In particular, $L|K^{\rac}$ is separable, where $K^{\rac}$ is the relative algebraic closure of $K$ in $L$.
\end{lem}
\begin{proof}
Since any algebraic extension of a perfect field is perfect and $K_1K^{1/p^\infty} | K^{1/p^\infty}$ is algebraic, it follows that $K_1K^{1/p^\infty}$ is a perfect field containing $K_1$. In particular, $K_1^{1/p^\infty} \subseteq K_1K^{1/p^\infty}$. The other inclusion is trivial. Hence, $K_1^{1/p^\infty} = K_1K^{1/p^\infty}$.
\begin{displaymath}
	\xymatrix {
		L & & & \\
		& & & K_1K^{1/p^\infty} & = & K_1^{1/p^\infty} \\
		K_1 \ar@{-}[uu] \ar@{-}[urrr]^{\txt{lin. disj.\;\;\;\;\;\;\;\;\;}}  & & & K^{1/p^\infty} \ar@{-}[u] \\
		K \ar@{-}[u] \ar@{-}[urrr]^{\txt{lin. disj.\;\;\;\;\;\;\;\;\;}} & & &
	}
\end{displaymath}
Since $L|K$ is separable, we know from Theorem~\ref{pindep} that $L|K$ is linearly disjoint from $K^{1/p^\infty}|K$. It follows that $L|K_1$ is linearly disjoint from $K_1K^{1/p^\infty}|K_1$ \cite[Proposition 3.1]{SL}. Combining this with the fact that  $K_1^{1/p^\infty} = K_1K^{1/p^\infty}$, we get that $L|K_1$ is separable.
\end{proof}

\begin{lem}		\label{perfhul}
If $L|K$ is a field extension such that $L^pK = L$, then $LK^{1/p^n} = L^{1/p^n}$ for all $n\in\bN$. In particular, $LK^{1/p^\infty} = L^{1/p^\infty}$.
\end{lem}
\begin{proof}
The condition $L = L^pK$ says that there is a basis $\cM_1$ of the extension $L|L^p$ in $K$. Moreover, this condition also gives $L^p = L^{p^2}K^p$, and thus, $L = L^{p^2}K^pK = L^{p^2}K$. Consequently, there is a basis $\cM_2$ of the extension $L|L^{p^2}$ in $K$. By a simple induction, there exists a basis $\cM_n$ of the extension $L|L^{p^n}$ in $K$ for every integer $n\ge 1$.

Let $x\in L^{1/p^n}$ be arbitrary. Then $x^{p^n}\in L$ and hence,
$$x^{p^n} = \sum_i y_i^{p^n} m_{i, n},$$
where $m_{i, n}\in\cM_n\subseteq K$ and $y_i\in L$. Consequently,
$$x = \sum_i y_i\, m_{i, n}^{1/p^n}.$$
In other words, $x\in LK^{1/p^n}$. Thus, $L^{1/p^n}\subseteq LK^{1/p^n}$. Since the other inclusion is obvious, we get that
$$LK^{1/p^n} = L^{1/p^n}.$$
Since $L^{1/p^\infty}$ and $K^{1/p^\infty}$ are direct limits of $L^{1/p^n}$ and $K^{1/p^n}$ respectively, for $n\ge 1$, we also obtain $LK^{1/p^\infty} = L^{1/p^\infty}.$
\end{proof}

For the proof of Theorem~\ref{sepdefextn}, we will need the following lemma.

\begin{lem}		\label{sepextnclass}
Let $K$ be a field of characteristic $p > 0$ and let $\cB$ be a $p$-basis of $K$. Let $L|K$ be a separable field extension. Then for any element $t\in L$ transcendental over $K$,
$$L|K(t) \mbox{ is separable if and only if } t\not\in L^p(\cB).$$
\end{lem}
\begin{proof}
First note that since $\cB$ is a $p$-basis of $K$ and $t$ is algebraically independent over $K$, it follows that $\cB\cup\{t\}$ is a $p$-basis of $K(t)$. We will now prove the contrapositive of the above statement.

If $t\in L^p(\cB)$, then $\cB\cup\{t\}$ is not $p$-independent in $L$. Thus, by Theorem~\ref{pindep}, the extension $L|K(t)$ is not separable.

Conversely, assume $L|K(t)$ is inseparable. Then by the equivalence of $(1)$ and $(4')$ in Theorem~\ref{pindep}, $\cB\cup\{t\}$ is not $p$-independent in $L$. However, since $L|K$ is separable, $\cB$ remains $p$-independent in $L$. By the Exchange Principle \cite[Lemma 2.7.1]{MFMJ}, it then follows that $t\in L^p(\cB)$.
\end{proof}

\subsection{Results on the structure of valued function fields}		\label{strvff}

We will say that a valued function field $(F|K,v)$ is \bfind{strongly inertially generated}, if there is a transcendence basis
\[
\cT = \{x_1, \ldots,x_r,y_1,\ldots, y_s\}
\]
of $(F|K,v)$ such that
\par\smallskip
a) \ $vF = vK(\cT) = vK\oplus\bZ vx_1\oplus \ldots\oplus \bZ vx_r$,\par
b) \ $y_1v,\ldots,y_sv$ form a separating transcendence basis of $Fv|Kv$,
\par\smallskip\noindent
and there is an element $a$ in some henselization $F^h$ of $(F,v)$ such that $F^h = K(\cT)^h(a)$, $va = 0$ and $K(\cT)v(av)|K(\cT)v$ is separable of degree equal to $[K(\cT)^h(a) : K(\cT)^h]$. The latter means that $F$ lies in the ``absolute inertia field'' of $(K(\cT), v)$.
As henselizations are separable algebraic extensions, this also means that the field extension $F|K$ is separably generated and hence separable.

\par\smallskip
For the purpose of the present paper, the importance of strongly inertially generated valued function fields lies in the the following embedding lemma \cite[Lemma~5.6]{FVK_TF}.

\begin{lem}                        \label{ael}
Let $(F|K, v)$ be a strongly inertially generated function field and $(K^*, v^*)$ be a henselian extension of $(K, v)$. Assume that $vF/vK$ is torsion free and that $Fv|Kv$ is separable. If $\rho:\>vF \longrightarrow v^*K^*$ is an embedding over $vK$ and $\sigma:\>Fv
\longrightarrow K^*v^*$ is an embedding over $Kv$, then there exists an embedding $\iota:\>(F, v) \longrightarrow (K^*, v^*)$ over $(K, v)$ that respects $\rho$ and $\sigma$, i.e., $v^*(\iota a) = \rho (va)$ and $(\iota a)v^*= \sigma(av)$ for all $a \in F$.
\end{lem}

\par\smallskip
We need a criterion for a valued function field to be strongly inertially generated. The following is Theorem~3.4 of \cite{HKFVK_AP}:

\begin{thm}                             \label{hrwtd}
Take a defectless field $(K, v)$ and a valued function field $(F|K, v)$ without transcendence defect. Assume that $Fv|Kv$ is a separable
extension and $vF/vK$ is torsion free. Then $(F|K, v)$ is strongly inertially generated. In fact, for each transcendence basis $\cT$
that satisfies conditions a) and b) there is an element $a$ with the required properties in every henselization of $F$.
\end{thm}

This theorem does in general not hold without the assumption that $(K,v)$ is a defectless field. However, if we replace this condition 
by the assumption that $F|K$ is separable, then we can still find a separating transcendence basis $\cT$ that satisfies conditions a) and b). This is the main assertion of Theorem~\ref{sepdefextn}. For its proof, we will need another lemma.

\begin{lem}		\label{sepdefextnlem}
Let $(F|K, v)$ be a separable valued function field with $v$ nontrivial on $F$. Then for any set of rationally independent values $\{\gamma_1, \ldots, \gamma_r\}$ in $vF$ over $vK$ and any set of algebraically independent residues $\{\bar{y}_1, \ldots, \bar{y}_s\}$ in $Fv$ over $Kv$, there exists a subset $\cT = \{x_1, \ldots, x_r, y_1, \ldots, y_s\}$ of $F$ with $vx_i = \gamma_i$, for $i = 1, \ldots, r$, and $y_jv = \bar{y}_j$, for $j = 1, \ldots, s$, such that $F|K(\cT)$ is separable.
\end{lem}
\begin{proof}
If char $K$ = 0, then char $K(\cT)$ = 0 and thus, every field extension of $K(\cT)$ is separable. So, we assume without loss of generality that char $K = p > 0$.


We prove the result by induction on $r + s$. If $r + s = 0$, then choosing $\cT = \emptyset$ does the job. If $r'\le r$, $s'\le s$ and $r'+s'+1= r+s$, then by induction hypothesis, there exists a subset $\cT'=\{x_1, \ldots, x_{r'}, y_1, \ldots, y_{s'}\}$ of $F$ with $vx_i = \gamma_i$ for $i = 1, \ldots, r'$ and $y_jv = \bar{y}_j$ for $j = 1, \ldots, s'$, such that $F|K(\cT')$ is separable. Replacing $K$ by $K(\cT')$, we have therefore reduced the problem to the case $r + s = 1$, which further splits into two subcases:


\par\smallskip
\underline{\bf Subcase I. $r = 0, s = 1$} : There is an element $\bar{y}_1\in Fv$ transcendental over $Kv$. Pick an element $z_1\in F$ such that $z_1v = \bar{y}_1$.
If $F|K(z_1)$ is separable, we set $y_1 := z_1$, and we are done.

Otherwise, $F|K(z_1)$ is inseparable. By \cite[Lemma 2.2]{FVK_TF}, $z_1$ is transcendental over $K$. Then $z_1\in F^p(\cB)$ by Lemma~\ref{sepextnclass}, where $\cB$ is a $p$-basis of $K$. Now pick an element $t_1\in F$ such that $F|K(t_1)$ is separable. We can do that because $F|K$ is a separable function field, and we can choose $t_1$ from a separating transcendence basis of $F|K$. Since $F|K(t_1)$ is separable, it follows again by Lemma~\ref{sepextnclass} that $t_1\not\in F^p(\cB)$. Pick $c\in F^p$ such that $vct_1 > 0$. This is possible because $vF^p$ is cofinal in $vF$ and $v$ is nontrivial on $F$. Since $c\in F^p$, we get that $ct_1\not\in F^p(\cB)$. Thus, $ct_1 + z_1\not\in F^p(\cB)$ either. By Lemma~\ref{sepextnclass} again, $F|K(ct_1 + z_1)$ is separable. Moreover, $(ct_1 + z_1)v = z_1v = \bar{y}_1$. We set $y_1 := ct_1 + z_1$, and we are done.

\par\smallskip
\underline{\bf Subcase II. $r = 1, s = 0$} :  There is an element $\gamma_1\in vF$ rationally independent over $vK$. Pick an element $z_1\in F$ such that $vz_1 = \gamma_1$. 
If $F|K(z_1)$ is separable, we set $x_1 := z_1$, and we are done.

Otherwise, $F|K(z_1)$ is inseparable, and $z_1\in F^p(\cB)$ by Lemma~\ref{sepextnclass} again. As before, pick an element $t_1\in F$ such that $F|K(t_1)$ is separable. It follows that $t_1\not\in F^p(\cB)$. For similar reasons as before, there is $c\in F^p$ such that $vct_1 > vz_1$. Since $c\in F^p$, we get that $ct_1\not\in F^p(\cB)$. Thus, $ct_1 + z_1\not\in F^p(\cB)$ either. By Lemma~\ref{sepextnclass} again, $F|K(ct_1 + z_1)$ is separable. Moreover, $v(ct_1 + z_1) = vz_1 = \gamma_1$. We set $x_1 := ct_1 + z_1$, and we are done.
\end{proof}


Now we prove Theorem~\ref{sepdefextn}.

\begin{proof}
If $v$ is trivial on $F$ (and hence on $K$), then $F\cong Fv$ and $K\cong Kv$. With these isomorphisms, we identify $F$ with $Fv$ and $K$ with $Kv$. Given $F|K$ separable, we choose a separating transcendence basis $\cT$ of $F|K$. By the above identification, $Fv|Kv$ is also separable and $\cT$ is a separating transcendence basis of $Fv|Kv$. Trivially, $vF = vK(\cT) = \{0\}$. Moreover, since $K(\cT)v\cong Kv(\cT)$, it follows that $Fv|K(\cT)v$ is separable.

Now suppose $v$ is nontrivial on $F$. Choose a transcendence basis $\{\bar{y}_1, \ldots, \bar{y}_s\}$ of $Fv|Kv$, and a maximal rationally independent set of elements $\{\gamma_1, \ldots, \gamma_r\}$ of $vF$ over $vK$. Choose a subset $\cT = \{x_1, \ldots, x_r, y_1, \ldots, y_s\}$ of $F$ as given by Lemma~\ref{sepdefextnlem}. By \cite[Lemma 2.2]{FVK_TF}, the set $\cT$ is algebraically independent over $K$ with
\begin{eqnarray*}
vK(\cT) & = & vK \oplus \bigoplus_{1\le i \le r} \bZ vx_i\;, \;\;\;\;\mbox{and}\\
K(\cT)v & = & Kv(y_1v, \ldots, y_sv).
\end{eqnarray*}
Since $(F|K, v)$ is without transcendence defect and $F|K(\cT)$ is separable, it follows that $\cT$ is in fact a separating transcendence basis of $F|K$.

Now assume that additionally $vF/vK$ is torsion free and $Fv|Kv$ is separable. Since $K(\cT)v = Kv(\bar{y}_1, \ldots, \bar{y}_s)$ and $Fv|K(\cT)v$ is finite, it follows that by choosing $\{\bar{y}_1, \ldots, \bar{y}_s\}$ to be a separating transcendence basis of $Fv|Kv$ to start with, we get that $Fv|K(\cT)v$ is separable. On the other hand, since $vF/vK(\cT)$ is finite and $vK(\cT)/vK$ is finitely generated, it follows that $vF/vK$ is also finitely generated. Thus,
$$vF = \Gamma \oplus \bigoplus_{1\le i\le r} \bZ \delta_i\;,$$
where $\delta_i\in vF$ for $1\le i \le r$, and $\Gamma/vK$ is finite. Since $vF/vK$ is torsion free by assumption, it follows that $\Gamma/vK$ is also torsion free. Combining this with the fact that $\Gamma/vK$ is finite and hence a torsion group, we get $\Gamma = vK$. Thus, by choosing the $\gamma_i$'s as the $\delta_i$'s to start with, we get that $vF = vK(\cT)$.
\end{proof}

One important application of the above theorem is the following result for separably defectless fields.

\begin{cor}		\label{sepdefless}
Let $(F|K, v)$ be a separable valued function field without transcendence defect. If, in addition, $(K, v)$ is a separably defectless field and $vK$ is cofinal in $vF$, then the extension $(F|K(\cT), v)$ is defectless, where $\cT$ is a standard valuation separating transcendence basis of $F|K$ as in Theorem~\ref{sepdefextn}.
\end{cor}
\begin{proof}
Since $(K(\cT)|K, v)$ is a valued function field without transcendence defect by Corollary~\ref{svtb->wtd}, $(K, v)$ is separably defectless and $vK$ is cofinal in $vK(\cT)$, it follows by \cite[Theorem 1.1]{FVK_ER} that $(K(\cT), v)$ is separably defectless. Moreover, by Theorem~\ref{sepdefextn}, the extension $F|K(\cT)$ is separable and finite. Consequently, $(F|K(\cT), v)$ is a defectless extension.
\end{proof}

\subsection{Algebra of Separably Tame Fields}		\label{sectastf}
We mention some of the algebraic properties of separably tame fields in this section, in addition to Theorem~\ref{stt3}. The details and the proofs of the statements can be found in \cite{FVK_thesis}, \cite{FVK_TF}.

As mentioned in the Introduction, it is an easy observation that separably tame fields of characteristic 0 are, in fact, tame fields. So, the interesting case is that of positive characteristic. Also, a separably tame field is always henselian and separably defectless since every finite separable-algebraic extension of a separably tame field is a tame, and thus defectless, extension. The converse is however not true; it needs additional assumptions on the value group and the residue field. Under the assumptions that we are going to use frequently, the converse will even hold for ``separable-algebraically maximal'' in place of ``henselian and separably defectless'', as mentioned in the following lemma.
\begin{lem}		\label{septame}
Take a nontrivially valued field $(K,v)$ of characteristic $p>0$. The following assertions are equivalent:
\begin{itemize}
\item  $(K,v)$ is separably tame,
\item  $(K,v)$ is separable-algebraically maximal, $vK$ is $p$-divisible and $Kv$ is perfect.
\end{itemize}
\end{lem}

As an immediate corollary, we get the following:

\begin{cor}		\label{samkap}
Every separable-algebraically maximal Kaplansky field is a separably tame field, but not conversely.
\end{cor}

The following result is crucial for our work because it provides a nice passage from separably tame fields to tame fields.

\begin{lem}		\label{phtd}
$(K, v)$ is a separably tame field if and only if $(K^{1/p^\infty}, v)$ is a tame field. In this case, if $v$ is nontrivial, then $(K, v)$ is dense in $(K^{1/p^\infty}, v)$.
\end{lem}

The following is another important lemma on separably tame fields that we will need in several instances.

\begin{lem}		\label{Xsrac}
Let $(L, v)$ be a separably tame field and $K\subseteq L$ a relatively algebraically closed subfield of $L$. If the residue field extension $Lv|Kv$ is algebraic, then $(K, v)$ is also a separably tame field, and moreover, $vL/vK$ is torsion free and $Lv = Kv$.
\end{lem}

\subsection{Model Theoretic Preliminaries}		\label{modprelim}
As mentioned in the Introduction, our primary language for talking about valued fields is $\cL_{\VF} := \{+, - , \cdot, ^{-1}, 0, 1, \cO\}$, where $\cO$ is a binary relation symbol for valuation divisibility, interpreted as follows:
$$\cO(x, y) :\iff vx\ge vy.$$
We prefer to write ``$vx\ge vy$'' in place of ``$\cO(x, y)$''. For convenience, we define the following relations:
\begin{eqnarray*}
vx > vy & \leftrightarrow & vx\ge vy \;\wedge\; \neg(vy\ge vx)\\
vx = vy & \leftrightarrow & vx\ge vy \;\wedge\; vy\ge vx.
\end{eqnarray*}
The definitions for the reversed relations $vx \le vy$ and $vx < vy$ are obvious.

It is an easy exercise to axiomatize the theory $T_{\VF}$ of valued fields in $\cL_{\VF}$, see \cite{FVK_TF}. The following facts are well-known; the easy proofs are left to the
reader.
\begin{lem}		\label{elprvgrf}
Take a valued field $(K,v)$.
\begin{itemize}
\item[a)] For every sentence $\varphi^{og}$ in the language of ordered groups, there is a sentence $\varphi^{og\rightarrow vf}$ in the language of valued fields such that for every valued field $(K, v)$, $\varphi^{og}$ holds in $vK$ if and only if $\varphi^{og\rightarrow vf}$ holds in $(K, v)$.
\item[b)] For every sentence $\varphi^r$ in the language of rings there is a sentence $\varphi^{r\rightarrow vf}$ in the language of valued fields such that for every valued field $(K,v)$, $\varphi^r$ holds in $Kv$ if and only if $\varphi^{r\rightarrow vf}$ holds in $(K, v)$.
\end{itemize}
\end{lem}

As immediate consequences of this lemma, we obtain:
\begin{cor}		\label{vgrfequiv}
If $(K,v)$ and $(L,v)$ are valued fields such that $(K,v)\equiv (L,v)$ in the language of valued fields, then $vK\equiv vL$ in the language of ordered groups, and $Kv\equiv Lv$ in the language of rings (and thus also in the language of fields). The same holds with $\prec$ or $\ec$ in place of $\equiv\,$.
\end{cor}

\begin{cor}\cite[Corollary 4.3]{FVK_TF}		\label{interpret}
If $(K, v)$ is $\kappa$-saturated for some cardinal $\kappa$, then so are $vK$ (in the language of ordered abelian groups) and $Kv$ (in the language of fields).
\end{cor}

But, as noticed in the Introduction, the language $\cL_{\VF}$ is not sufficient for purposes of separably tame fields because we need to deal with separable extensions only. Since, in characteristic 0 every field extension is separable, we need to deal with the positive characteristic case. So, for the rest of this section, let $p$ be a fixed prime; and by $(K, v)$ we will denote a valued field of characteristic $p$. We choose a language in which every field extension is a $p$-basis extension, i.e., linear independence over the subfield of $p^{th}$ powers is preserved. Let
$$\cL_\rQ := \cL_{\VF} \cup \{Q_m\}_{m = 1}^\infty$$
where $Q_m$ is an $m$-ary predicate symbol for each $m \ge 1$. In a valued field $(F, v)$ of characteristic zero, the $Q_m$'s are interpreted trivially, i.e., $Q_m(x_1, \ldots, x_m)$ holds for all $x_1, \ldots, x_m\in F$. But in $(K, v)$, a valued field of characteristic $p > 0$, the $Q_m$'s are interpreted as follows: for any $m$-tuple $\{x_1, \ldots, x_m\}$ from $K$,
\begin{eqnarray*}
Q_m(x_1, \ldots, x_m)\mbox{ holds } & \iff & \mbox{the elements } x_1, \ldots, x_m \mbox{ are $p$-independent}\\
& \iff & \mbox{the monomials of exponents $< p$ in the $x_i$'s are} \\
& & \mbox{linearly independent over the subfield of $p^{th}$ powers.}
\end{eqnarray*}
We add these defining axioms for the predicates $Q_m$ (for $m \ge 1$) to $T_{VF}$ and get the theory $T_{\rQ}$ of valued fields of a fixed characteristic in the language $\cL_\rQ$. Note that $T_\rQ$ is an $\forall\exists$-theory.

It is now easy to see that if $(L|K, v)$ is an extension in the language $\cL_\rQ$, then any $p$-independent subset in $K$ (in particular any $p$-basis of $K$) remains $p$-independent in $L$, and hence by Theorem~\ref{pindep}, the extension $L|K$ is separable.

It is also easy to see that if $[K : K^p] \ge p^m$ for some $m \ge1$, then there are elements $a_1, \ldots, a_m\in K$ such that $Q_m(a_1, \ldots, a_m)$ holds. Thus, for any fixed $e\in\bN$, we can elementarily define the class of all valued fields of characteristic $p$ and $p$-degree at most $e$ by adding to $T_\rQ$ the axiom
$$\forall x_1\cdots\forall x_{e+1}\, \neg Q_{e+1}(x_1, \ldots, x_{e+1}),$$
and the class of all valued fields of $p$-degree exactly $e$ by further adding the axiom
$$\exists x_1\cdots\exists x_e\, Q_e(x_1, \ldots, x_e).$$

Finally, note that since the predicates $Q_m$, for $m\ge 1$, are definable in the language $\cL_{\VF}$, the two languages $\cL_{\VF}$ and $\cL_\rQ$ give the same definable sets. In particular, for any three valued fields $(K, v)$, $(F, v)$ and $(L, v)$ of the same positive characteristic $p$,
$$(F, v)\equiv_{(K, v)}(L, v) \mbox{ in } \cL_{\VF} \iff (F, v)\equiv_{(K, v)}(L, v) \mbox{ in } \cL_\rQ.$$
In particular, if $(K, v)\subseteq (L, v)$, then
$$(K, v) \prec (L, v) \mbox{ in } \cL_{\VF} \iff (K, v) \prec (L, v) \mbox{ in } \cL_\rQ.$$

Quite naturally the analogous result with $\prec$ replaced by $\prec_{\exists}$ does not hold in general. It does hold in the characteristic zero case. However, even in the positive characteristic case, it holds if the extension $L|K$ share a common $p$-basis as the following lemma shows; for a proof, see \cite[Lemme 3.12]{FD_thesis}.
\begin{lem} 		\label{VF2Q}
Let $(L|K, v)$ be an extension of valued fields with a common $p$-basis. Then
$$(K, v)\prec_\exists (L, v)\; \mathrm{in}\; \cL_{\VF}\iff (K, v)\prec_\exists (L, v) \;\mathrm{in}\; \cL_\rQ.$$
\end{lem}

\section{Embedding Lemmas}
\subsection{Necessary conditions for the AKE$^\exists$ Principle}
In this section we discuss tools for the proof of AKE$^\exists$ Principle.

We will need a model theoretic tool which we will apply to valued fields as well as value groups and residue fields. We consider a countable language $\cL$ and $\cL$-structures ${\eu B}$ and ${\eu A}^*$ with a common substructure ${\eu A}$. We will say that $\sigma$ is an \bfind{embedding of ${\eu B}$ in ${\eu A}^*$ over ${\eu A}$} if it is an embedding of ${\eu B}$ in ${\eu A}^*$ that leaves the universe $A$ of ${\eu A}$ elementwise fixed.

\begin{prop} 		\label{ec}
Let ${\eu A}\subseteq {\eu B}$ and ${\eu A}\subseteq {\eu A}^*$ be extensions of $\cL$-structures. If\/ ${\eu B}$ embeds over
${\eu A}$ in ${\eu A}^*$ and if\/ ${\eu A}\ec {\eu A}^*$, then ${\eu A} \ec {\eu B}$. Conversely, if ${\eu A}\ec {\eu B}$ and if ${\eu A}^*$ is $|B|^+$-saturated, then ${\eu B}$ embeds over ${\eu A}$ in ${\eu A}^*$.
\end{prop}

For a proof of this, see \cite[Proposition 5.1]{FVK_TF}. If we have an extension ${\eu A}\subseteq {\eu B}$ of $\cL$-structures and want to show that ${\eu A}\ec {\eu B}$, then by the above proposition it suffices to show that ${\eu B}$ embeds over ${\eu A}$ in some elementary extension ${\eu A}^*$ of ${\eu A}$. This is the motivation for {\bf embedding lemmas},\index{embedding lemma} which will play an important role later in our paper. When we look for such embeddings, we can use a very helpful principle which follows immediately from the previous proposition because ${\eu A} \ec {\eu B}$ if and only if ${\eu A}\ec {\eu B}_0$ for every substructure ${\eu B}_0$ of ${\eu B}$ which is finitely generated over ${\eu A}$ (as every existential sentence only talks about finitely many elements).

\begin{lem}		\label{embfingen}
Let ${\eu A}\subseteq {\eu B}$ and ${\eu A}\subseteq {\eu A}^*$ be extensions of $\cL$-structures. Assume that ${\eu A}^*$ is
$|B|^+$-saturated. If every substructure of ${\eu B}$ which is finitely generated over ${\eu A}$ embeds over ${\eu A}$ in ${\eu A}^*$, then also ${\eu B}$ embeds over ${\eu A}$ in ${\eu A}^*$.
\end{lem}

We will also need the following well known facts (which were proved, e.g., in L.~van den Dries' thesis \cite{LD_thesis}).

\begin{lem}		\label{sica}
a)\ \ Take an extension $G|H$ of torsion free abelian groups. Consider it as an extension of $\cL_{\rm G}$-structures, where $\cL_{\rm G} = \{+, -, 0\}$ is the language of abelian groups. If $H$ is existentially closed in $G$ in the language $\cL_{\rm G}$, then $G/H$ is torsion free.

b)\ \ Take a field extension $L|K$. If $K$ is existentially closed in $L$ in the language $\cL_{\rm F}$ of fields (or in the language $\cL_{\rm R}$ of rings), then $L|K$ is regular, i.e., $K$ is relatively algebraically closed in $L$ and $L|K$ is separable.
\end{lem}

\subsection{Separable extensions without transcendence defect.}			\label{sectmwtd}
Our main goal in this section is to prove Theorem~\ref{sewtd} by constructing an appropriate embedding that respects the corresponding embeddings at the level of value groups and residue fields. We now make this precise. Take a separably tame field $(K,v)$ and a separable extension $(L|K,v)$ without transcendence defect. We choose $(K^*,v^*)$ to be an $|L|^+$-saturated elementary extension of $(K,v)$. Since ``henselian'' is an elementary property, $(K^*,v^*)$ is henselian like $(K,v)$. Further, it follows from Corollary~\ref{interpret} that $K^*v^*$ is an $|L|^+$- (and consequently, $|Lv|^+$-) saturated elementary extension of $Kv$ and that $v^*K^*$ is an $|L|^+$- (and consequently, $|vL|^+$-) saturated elementary extension of $vK$. Assume that the side conditions $vK\ec vL$ and $Kv\ec Lv$ hold. We wish to prove that $(K,v)\ec (L,v)$ in the language $\cL_{\VF}$.

\begin{rmk}
As noted in the Introduction, this claim is trivially true if $(K, v)$ (and hence $(L, v)$) is trivially valued. So, we can assume without loss of generality that $(K, v)$ is nontrivially valued. However, later on while proving the Separable Relative Embedding Property in Section~\ref{relEmb}, we will need to extend embeddings over trivially valued separably tame fields as well. So, we handle the general case here. For the rest of this proof, we allow $v$ to be trivial on $K$, but assume that $(K^*, v^*)$ is a nontrivially valued $|L|^+$-saturated separably tame field extending $(K, v)$.
\end{rmk}

By our assumptions, $vK\ec vL$ and $Kv\ec Lv$. Thus, by Lemma~\ref{sica}, $vL/vK$ is torsion free and $Lv|Kv$ is separable. Also, by Proposition~\ref{ec}, there exist embeddings
$$\rho:\>vL \longrightarrow v^*K^*$$
over $vK$ and
$$\sigma:\;Lv \longrightarrow K^*v^*$$
over $Kv$. Here, the embeddings of value groups and residue fields are understood to be (injective) homomorphisms of ordered groups and of fields respectively.

Our goal is to prove that $(K,v)\ec (L,v)$ in the language $\cL_{\VF}$. By Proposition~\ref{ec}, this can be achieved by showing the existence of an embedding
$$\iota:\> (L,v) \longrightarrow (K^*,v^*)$$
over $K$, i.e., an embedding of $L$ in $K^*$ over $K$ preserving the valuation, that is,
$$\forall x\in L:\; x\in\cO_L\Longleftrightarrow\iota x \in\cO_{K^*}\;.$$

According to Lemma~\ref{embfingen}, such an embedding exists already if it exists for every finitely generated subextension $(F|K,v)$ of $(L|K,v)$. In this way, we reduce our embedding problem to an embedding problem for valued algebraic function fields $(F|K,v)$. Since in the present case, $(L|K, v)$ is assumed to be a separable extension without transcendence defect, the same holds for every finitely generated subextension $(F|K, v)$. Also, $vF/vK$ is torsion free, $Fv|Kv$ is separable, and there are corresponding embeddings $\rho$ and $\sigma$, which are just the restrictions of the original $\rho$ and $\sigma$ to $vF$ and $Fv$ respectively.

We will construct the embedding $\iota$ from the corresponding embedding lemma \cite[Lemma 5.6]{FVK_TF}. To that end, we consider the extension $(FK^{1/p^\infty}|K^{1/p^\infty},$ $v)$. Clearly it is finitely generated because $F|K$ is. Since $K^{1/p^\infty}|K$ is an algebraic extension, it follows that $\trdeg F|K = \trdeg FK^{1/p^\infty}|K^{1/p^\infty}$. Since $(K, v)$ is separably tame, Lemma~\ref{phtd} yields that $(K^{1/p^\infty}, v)$ is tame, and hence defectless. The residue field extension $FK^{1/p^\infty}v|K^{1/p^\infty}v$ is separable because $K^{1/p^\infty}v = (Kv)^{1/p^\infty}$ by Lemma~\ref{sep_perf_hull}, and hence $K^{1/p^\infty}v$ is perfect. The value group extension $vFK^{1/p^\infty}|vK^{1/p^\infty}$ is torsion free because $vK^{1/p^\infty} = 1/p^\infty\,vK$ by Lemma~\ref{sep_perf_hull}, and hence $vK^{1/p^\infty}$ is $p$-divisible; moreover, $vF/vK$ is torsion free and
$$vFK^{1/p^\infty}\subseteq vF^{1/p^\infty} = \dfrac{1}{p^\infty}vF.$$
Since $(K^*, v^*)$ is also separably tame, its perfect hull $({K^*}^{1/p^\infty}, v^*)$ is tame, and hence henselian. Moreover, the embeddings $\rho: vF\longrightarrow v^*K^*$ over $vK$ and $\sigma: Fv\longrightarrow K^*v^*$ over $Kv$ can be uniquely extended to embeddings (also denoted by) $\rho: vF^{1/p^\infty}\longrightarrow v^*{K^*}^{1/p^\infty}$ over $vK^{1/p^\infty}$ and $\sigma: F^{1/p^\infty}v\longrightarrow {K^*}^{1/p^\infty}v^*$ over $K^{1/p^\infty}v$. In particular, there are embeddings $\rho: vFK^{1/p^\infty}\longrightarrow v^*{K^*}^{1/p^\infty}$ over $vK^{1/p^\infty}$ and $\sigma: FK^{1/p^\infty}v\longrightarrow {K^*}^{1/p^\infty}v^*$ over $K^{1/p^\infty}v$.

Finally, note that if $(K, v)$ is nontrivially valued, Lemma~\ref{septame} yields that $vK$ is $p$-divisible and hence $vK^{1/p^\infty} = vK$. On the other hand, if $(K, v)$ is trivially valued, then so is $(K^{1/p^\infty}, v)$, and again we have $vK^{1/p^\infty} = vK$. In either case, since $(Kv)^{1/p^\infty}|Kv$ is algebraic, we have
$$\trdeg Fv|Kv = \trdeg Fv(Kv)^{1/p^\infty}|(Kv)^{1/p^\infty} \le \trdeg FK^{1/p^\infty}v|K^{1/p^\infty}v.$$
As a result, it follows by (\ref{wtdgeq}) and the fact that $(F|K, v)$ is without transcendence defect,
\begin{eqnarray*}
\trdeg FK^{1/p^\infty}|K^{1/p^\infty} & \ge & \trdeg FK^{1/p^\infty}v|K^{1/p^\infty}v \;+\; \mathrm{r.r.}\; v FK^{1/p^\infty}/vK^{1/p^\infty}\\
& \ge & \trdeg Fv|Kv \;+\; \mathrm{r.r.}\; vF/vK\\
& = & \trdeg F|K \\
& = & \trdeg FK^{1/p^\infty}|K^{1/p^\infty}
\end{eqnarray*}
and hence there is equality throughout. In particular, the extension $(FK^{1/p^\infty}|K^{1/p^\infty}, v)$ is also without transcendence defect.

Since $(K^{1/p^\infty}, v)$ is defectless, it follows from Theorem~\ref{hrwtd} that $(FK^{1/p^\infty}|K^{1/p^\infty}, v)$ is strongly inertially generated. Hence, from Lemma~\ref{ael} we obtain an embedding
\[
j: (FK^{1/p^\infty}, v)\longrightarrow ({K^*}^{1/p^\infty}, v^*)
\]
over $K^{1/p^\infty}$ that respects $\rho$ and $\sigma$, i.e., $v^*j(a) = \rho (va)$ and $j(a)v^*= \sigma(av)$ for all $a \in FK^{1/p^\infty}$. Since $(K^*, v^*)$ is dense in $({K^*}^{1/p^\infty}, v^*)$ by Lemma~\ref{phtd}, it follows that ${K^*}^{1/p^\infty}\subseteq {K^*}^c$, where $({K^*}^c, v^*)$ is the completion of $(K^*, v^*)$. By taking the restriction of $j$ to $(F, v)$, we thus get an embedding
$$j: (F, v) \longrightarrow ({K^*}^c, v^*)$$
over $K$ that respects $\rho$ and $\sigma$.
We will use this embedding $j$ to construct our required embedding $\iota$ of $(F, v)$ in $(K^*, v^*)$ over $K$ that also respects $\rho$ and $\sigma$. The case of such valued function fields is covered by the following embedding lemma.

\begin{lem}		\label{EI}
{\bf (Embedding Lemma I)}\newline
Let $(K, v)$ be a valued field (the valuation is allowed to be trivial), $(F|K, v)$ a separable valued function field without transcendence defect, and $(K^*, v^*)$ a henselian extension of $(K, v)$. Also, assume that $vF/vK$ is torsion free, $Fv|Kv$ is separable, and there are embeddings $\rho:\> vF \longrightarrow v^*K^*$ over $vK$ and $\sigma:\>Fv \longrightarrow K^*v^*$ over $Kv$. If there is an embedding $j: (F, v)\longrightarrow ({K^*}^c, v^*)$ over $K$ that respects $\rho$ and $\sigma$, then there exists an embedding $\iota:\>(F, v) \longrightarrow (K^*, v^*)$ over $K$ that also respects $\rho$ and $\sigma$.
\end{lem}

To prove this embedding lemma, it suffices to prove the following.
\begin{lem}		\label{EILem}
Let $(K, v)$ be a valued field (the valuation is allowed to be trivial) and $(K^*, v^*)$ be a henselian extension of $(K, v)$. Let $(F_1|K, v^*)$ be a separable valued function field without transcendence defect contained in ${K^*}^c$. Assume that $v^*F_1/vK$ is torsion free and $F_1v^*|Kv$ is separable. Then there exists an embedding $\iota_1:\>(F_1, v^*) \longrightarrow (K^*, v^*)$ over $K$ such that $v^*\iota_1(a) = v^*a$ and $\iota_1(a)v^* = av^*$ for all $a\in F_1$.
\end{lem}


The word ``neighborhood'' in the following proof refers to open sets in the valuation topology. For a polynomial $f\in \cO_{F_1}[X_1, \ldots, X_n]$ (for $n\in\bN$), we denote by $fv^*$ the polynomial in $F_1v^*[X_1, \ldots, X_n]$ that is obtained from $f$ by replacing all its coefficients by their residues. And for a polynomial $f\in F_1[X_1, \ldots, X_n]$ (for $n\in\bN$), we denote by $\iota_1(f)$ the polynomial in $K^*[X_1, \ldots, X_n]$ that is obtained from $f$ by replacing all its coefficients by their images under $\iota_1$.

\begin{proof}[Proof of Lemma~\ref{EILem}]
By Theorem~\ref{sepdefextn}, there is a separating transcendence basis $\cT := \{x_1, \ldots, x_r, y_1, \ldots,$ $y_s\}$ of $F_1|K$, such that the values $\{v^*x_i \mid 1\le i\le r\}$ are rationally independent over $vK$ with $v^*K(\cT) = v^*F_1$, and the residues $\{y_jv^* \mid 1\le j\le s\}$ form a separating transcendence basis of $F_1v^*|Kv$ with $F_1v^*|K(\cT)v^*$ separable and finite. Since $F_1|K(\cT)$ is finite and separable, by the Primitive Element Theorem, there is $b\in F_1$ such that $F_1 = K(\cT, b)$. Let $\widetilde{h}(Z) \in K(\cT)[Z]$ be the minimal polynomial of $b$ over $K(\cT)$. Multiplying $\widetilde{h}$ by a suitable element of $K[\cT]$ to clear the denominators of the coefficients, we obtain a polynomial in $K[\cT][Z]$. This polynomial can be written as  $h(x_1, \ldots, x_r, y_1, \ldots, y_s, Z)$ for some polynomial $h\in K[X_1, \ldots, X_{r+s}, Z]$. In particular, we have
\begin{eqnarray*}
h(x_1, \ldots, x_r, y_1, \ldots, y_s, b) & = & 0, \;\;\;\;\;\mbox{and} \\
\dfrac{\partial h}{\partial Z}(x_1, \ldots, x_r, y_1, \ldots, y_s, b) & \not= & 0.
\end{eqnarray*}
Also, since $F_1|K(\cT)$ is finite, there are only finitely many possible extensions of the valuation $v^*$ from $K(\cT)$ to $F_1$, say $v^*_1, v^*_2, \ldots, v^*_g$ with $v^* = v^*_1$. Choose witnesses $a_\ell \in F_1$ such that $a_\ell\in\cO_{v^*}\setminus\cO_{v^*_\ell}$ for $\ell = 2, \ldots, g$. Since $F_1v^*|K(\cT)v^*$ is finite and separable, by the Primitive Element Theorem again, there is $\bar{a}_1\in F_1v^*$ such that $F_1v^* = K(\cT)v^*(\bar{a}_1)$. Choose $a_1\in F_1$ such that $a_1v^* = \bar{a}_1$. These elements $x_1, \ldots, x_r, y_1, \ldots, y_s, b, a_1, a_2, \ldots, a_g$ are the key players in our proof. Note that since $F_1 = K(\cT, b)$, there are rational functions $h_\ell\in K(X_1, \ldots, X_{r+s}, Z)$ such that $a_\ell = h_\ell(x_1, \ldots, x_r, y_1, \ldots, y_s, b)$ for $\ell = 1, \ldots, g$. Using the density of $(K^*, v^*)$ in $({K^*}^c, v^*)$, we wish to find elements $x_1', \ldots,$ $x_r', y_1', \ldots, y_s', b', a_1', \ldots, a_g'$ in $K^*$ that are ``close enough'' to the corresponding elements in ${K^*}^c$ and can be used to define our required embedding $\iota_1$.

To start with, we fix a small enough neighborhood $U$ of zero such that for all $x_i'\in x_i + U$ and $y_k'\in y_k + U$ we have $v^*x_i' = v^*x_i$ and $y_k'v^* = y_kv^*$, for $1\leq i \leq r$, $1\leq k\leq s$. By the continuity of rational functions, we can even choose $U$ so small that if $b'\in b+U$, then $a_\ell':=h_\ell(x_1', \ldots, x_r', y_1', \ldots, y_s', b')$ belongs to $h_\ell(x_1, \ldots, x_r, y_1, \ldots, y_s, b) + {\mathcal M}=a_\ell + {\mathcal M}$,  for $1\le\ell\le g$. For such $a_\ell'$, it follows that $a_1'v^*=a_1v^*$, and
$v^*a_\ell' \ge 0$ since $v^*a_\ell\ge 0$.

Since $(K^*, v^*)$ is dense in $({K^*}^c, v^*)$, the set $U\cap K^*$ is nonempty and therefore a neighborhood of zero in $K^*$. By the Implicit Function Theorem \cite[Theorem 7.4]{APMZ}, applied to the henselian field $(K^*, v^*)$, there is a neighborhood $V$ of zero in $K^*$, which we can take to be a subset of $U\cap K^*$, such that for $x_i'\in x_i + V$, $1\le i\le r$, and $y_k'\in y_k + V$, $1\le k\le s$, there is $b'\in b + U\cap K^*$ such that $h(x_1', \ldots, x_r', y_1', \ldots, y_s', b') = 0$.

The values $v^*x_1', \ldots, v^*x_r'$ are rationally independent over $vK$ since the same holds for the values $v^*x_1, \ldots, v^*x_r$, and the residues $y_1'v^*, \ldots, y_s'v^*$ are algebraically independent over $Kv$ since the same holds for the residues $y_1v^*, \ldots, y_sv^*$. Consequently, by \cite[Lemma 2.2]{FVK_TF}, the elements of the set $\cT' := \{x_1',\ldots, x_r', y_1', \ldots, y_s'\}$ are algebraically independent over $K$, and the map
\[
x_i\mapsto x_i'\;,\;\;\;y_k\mapsto y_k'\;,\;\;\;1\leq i\leq r\,,\; 1\leq k\leq s,
\]
induces an isomorphism $\iota_1: K(\cT) \rightarrow K(\cT')$. Furthermore, for every $f\in K[\cT]$, written as in \cite[Lemma 2.2]{FVK_TF},
\begin{eqnarray*}
v^*\iota_1(f) & = & \min_k\left(v^*c_k\,+\,\sum_{1\leq i\leq r} \mu_{k,i}v^*x'_i\right) \>=\> \min_k\left(vc_k\,+\,\sum_{1\leq i\leq r}\mu_{k,i} v^*x_i\right) = v^*f
\end{eqnarray*}
showing that $\iota_1$ satisfies $v^*\iota_1(a) = v^*a$ for all $a\in K(\cT)$. If $v^*f = 0$, then
\[
fv^*\>=\>\left(\displaystyle\sum_{\ell}^{} c_{\ell}\,\prod_{1\leq j\leq s} y_j^{\nu_{\ell,j}}\right)v^*\>=\>\displaystyle\sum_{\ell}^{} (c_{\ell}v)\prod_{1\leq j\leq s} (y_j v^*)^{\nu_{\ell,j}}
\]
where the sum runs only over those $\ell=k$ for which $\mu_{k,i}=0$ for all $i$, and a similar formula holds for $\iota_1(f)v^*$ with the same indices $\ell$. Hence,
\begin{eqnarray*}
\iota_1(f)v^* & = & \displaystyle\sum_{\ell}^{} (c_{\ell}v^*)\prod_{1\leq j\leq s} (y_j' v^*)^{\nu_{\ell,j}} \>=\>\displaystyle\sum_{\ell}^{} (c_{\ell}v)\prod_{1\leq j\leq s} (y_j v^*)^{\nu_{\ell,j}} =  fv^*
\end{eqnarray*}
showing that $\iota_1$ satisfies $\iota_1(a)v^* = av^*$ for all $a\in K(\cT)$. Hence, the isomorphism preserves valuation and residue map.

The assignment
$$b\mapsto b'$$
extends $\iota_1$ to a field isomorphism from $F_1$ onto $K(\cT', b')$ over $K(\cT)$ which sends $a_\ell$ to $a_\ell'$. Since $v^*\circ\iota_1$ is a valuation on $F_1$ and since there are only $g$ many extensions of $v^*$ from $K(\cT)$ to $F_1$, we get that
$$v^* \circ \iota_1 = v^*_{\ell_0}$$
for some $1\le\ell_0\le g$. Since $v^*_{\ell_0}(a_{\ell_0}) = v^*(\iota_1(a_{\ell_0})) = v^*(a_{\ell_0}')\ge 0$, we must have that $\ell_0 = 1$. Hence, $v^*\circ\iota_1 = v^*_1 = v^*$, that is, $\iota_1$ is a valued field isomorphism from $(F_1, v^*)$ onto $(K(\cT', b'), v^*)$ over $K(\cT)$.

Finally, recall that $a_1'v^* = a_1v^*$, and hence we obtain $\iota_1(a_1)v^* = a_1v^*$. Moreover, for each $a\in F_1$, we have $av^* = \bar{f}(\bar{a}_1)$ for some $\bar{f}(X)\in K(\cT)v^*[X]$. We choose $f(X)\in K(\cT)[X]$ of the same degree as $\bar{f}$ such that $fv^* = \bar{f}$. Then $v^*(a - f(a_1)) > 0$. Hence, $v^*\iota_1(a - f(a_1)) > 0$. Therefore,
$$\iota_1(a)v^* = \iota_1(f(a_1))v^* = \iota_1(f)(\iota_1(a_1))v^* = \iota_1(f)v^*(\iota_1(a_1)v^*) = fv^*(a_1v^*) = \bar{f}(\bar{a}_1) = av^*.$$
Hence, $\iota_1$ is the required embedding of $(F_1, v^*)$ in $(K^*, v^*)$ over $K$.
\end{proof}

We return to the proof of Theorem~\ref{sewtd}. We have already shown that for every finitely generated subextension $(F|K, v)$ of $(L|K, v)$, there is an embedding $j$ of $(F, v)$ in $({K^*}^c, v^*)$ over $K$ that respects $\rho$ and $\sigma$. Then by Lemma~\ref{EI}, there is an embedding $\iota$ of $(F, v)$ in $(K^*, v^*)$ over $K$ that respects $\rho$ and $\sigma$. Since this holds for every finitely generated subextension $(F|K, v)$ of $(L|K, v)$, it follows from Lemma~\ref{embfingen} that also $(L, v)$ embeds in $(K^*, v^*)$ over $K$. By Proposition~\ref{ec}, this shows that $(K, v)$ is existentially closed in $(L, v)$ in the language $\cL_{\VF}$, and thus we have proved Theorem~\ref{sewtd}.

For future use, we make our result more precise. This is proved by a standard application of saturation. For a demonstration of how saturation is used in this context, see \cite[Lemma 5.7]{FVK_TF}.

\begin{lem} {\bf (Embedding Lemma II)}		\label{EII}\newline
Take a valued field $(K, v)$ (the valuation is allowed to be trivial), a separable extension $(L|K, v)$ without transcendence defect, and an $|L|^+$-saturated henselian extension $(K^*, v^*)$ of $(K, v)$. Also, assume that $vL/vK$ is torsion free, $Lv|Kv$ is separable, and there are embeddings $\rho:\; vL \longrightarrow v^*K^*$ over $vK$ and $\sigma:\; Lv\longrightarrow K^*v^*$ over $Kv$. If for every finitely generated subextension $(F|K, v)$ of $(L|K, v)$, there is an embedding
$$j: (F, v)\longrightarrow({K^*}^c, v^*)$$
over $K$ that respects $\rho$ and $\sigma$, then there exists an embedding
$$\iota:\>(L, v) \longrightarrow (K^*, v^*)$$
over $K$ that respects $\rho$ and $\sigma$.
\end{lem}

Before we end this section, we state another useful embedding result about separable-algebraically maximal fields. For that, we need the following general Embedding Lemma.

\begin{lem} {\bf (Embedding Lemma III)}\cite[Lemma 6.2]{FVK_TF}		\label{EIII} \newline
Let $(K(x)|K,v)$ be a nontrivial immediate extension of valued fields. If $x$ is the limit of a pseudo-Cauchy sequence of
transcendental type in $(K,v)$, then $(K(x),v)^h$ embeds over $K$ in every $|K|^+$-saturated henselian extension $(K,v)^{*}$ of $(K,v)$.
\end{lem}

\par\medskip
The model theoretic application of Embedding Lemma~III is:

\begin{cor}		\label{imrf}
Let $(K,v)$ be a henselian field and $(K(x)|K,v)$ an immediate extension such that $x$ is the limit of a pseudo-Cauchy sequence of transcendental type in $(K,v)$. Then $(K,v)\ec (K(x),v)^h$ in $\cL_{\VF}$.
\end{cor}
\begin{proof}
Choose $(K,v)^*$ to be a $|K|^+$-saturated elementary extension of $(K,v)$. Since ``henselian'' is an elementary property, $(K,v)^*$ will also be henselian. Apply Embedding Lemma~III and Proposition~\ref{ec}.
\end{proof}

We use this to prove the following result about separable-algebraically maximal fields.
\begin{prop}		\label{sam}
A separable-algebraically maximal field is existentially closed in $\cL_{\VF}$ in every henselization of an immediate rational function field of transcendence degree $1$.
\end{prop}

To prove this proposition, we need one more lemma first.
\begin{lem}		\label{pcc}
Let $(K, v)$ be separable-algebraically maximal. Any pseudo-Cauchy sequence in $K$ of algebraic type is automatically a Cauchy sequence and hence has a unique limit.
\end{lem}
\begin{proof}
Suppose $(x_\alpha)_{\alpha < \lambda}$ (with $\lambda$ a limit ordinal) is a pseudo-Cauchy sequence of algebraic type which is not Cauchy.

By Kaplansky \cite[Theorem 3]{IK}, there is an immediate extension $(K(a)|K, v)$ such that $a$ is a pseudo limit of the sequence $(x_\alpha)$ and $P(a) = 0$, where $P(x)$ is a minimal polynomial of $(x_\alpha)$. By assumption, $a\not\in K^c$. Now consider the following diagram.

\begin{displaymath}
	\xymatrix {
		& K(a)^c & \\
		& K^c(a) \ar@{-}[u] & \\
		K(a) \ar@{-}[ur] & & K^c \ar@{-}[ul] \\
		& K \ar@{-}[ul] \ar@{-}[ur] &
	}
\end{displaymath}

Since $(K(a)|K, v)$ is immediate, in particular $vK$ is cofinal in $vK(a)$, we know that $K^c\subseteq K(a)^c$. Since $(K^c|K, v)$, $(K(a)^c|K(a), v)$ and $(K(a)|K, v)$ are immediate extensions, it follows that the extension $(K(a)^c|K^c, v)$, and hence the subextension $(K^c(a)|K^c, v)$, is immediate. Moreover, $K^c(a)|K^c$ is an algebraic extension. But $(K^c, v)$ is algebraically maximal, since $(K, v)$ is separable-algebraically maximal \cite[Corollary 6.8]{FVK_AS}. Thus, $(K^c, v)$ does not admit a proper immediate algebraic extension, which implies that $K^c(a) = K^c$. In other words, $a\in K^c$, which gives the required contradiction.
\end{proof}

Now we are ready to prove Proposition~\ref{sam}.
\begin{proof}
Let $(K, v)$ be a separable-algebraically maximal valued field, and $(K(x)|K, v)$ be a proper immediate extension of transcendence degree 1. By \cite[Theorem 1]{IK}, there exists a pseudo-Cauchy sequence $(x_\alpha)_{\alpha < \lambda}$ in $K$ without a pseudo limit in $K$ such that $x$ is a pseudo limit of $(x_\alpha)$.

If $(x_\alpha)$ were of algebraic type, then by \cite[Theorem 3]{IK} again, there would be an immediate algebraic extension $K(y)$ of $K$ with $y$ a pseudo limit of $(x_\alpha)$. But by Lemma~\ref{pcc}, $(x_\alpha)$ is Cauchy, and has a unique limit. Hence, $x = y$. In particular, $x$ would be algebraic over $K$, which gives a contradiction.

Thus, $(x_\alpha)$ is of transcendental type. Since a separable-algebraically maximal field is henselian, the result now follows by Corollary~\ref{imrf}.
\end{proof}

\section{The Separable Relative Embedding Property}			\label{relEmb}
Inspired by the assertion of Lemma~\ref{EI}, we define a property that will play a key role in our approach to the model theory of separably tame fields. Let {\bf C} be a class of valued fields. We will say that {\bf C} has the \bfind{Separable Relative Embedding Property}, if the following holds:
\par\smallskip\noindent
if $(L,v),(K^*,v^*)\in {\bf C}$ with common subfield $(K,v)$ such that
\newline
$\bullet$ \ $(K, v)$ is separably tame, \newline
$\bullet$ \ $L|K$ is separable, \newline
$\bullet$ \ $(L, v)$ is $\aleph_0$-saturated and $(K^*,v^*)$ is $|L|^+$-saturated, \newline
$\bullet$ \ $vL/vK$ is torsion free and $Lv|Kv$ is separable, \newline
$\bullet$ \ there are embeddings $\rho:\; vL \longrightarrow v^*K^*$ over $vK$ and $\sigma:\; Lv \longrightarrow K^*v^*$ over $Kv$, \newline
then there exists an embedding $\iota:\>(L, v) \longrightarrow (K^*, v^*)$ over $K$ that respects $\rho$ and $\sigma$.

\par\smallskip
We now define another property of {\bf C} which is very important for our purposes. If ${\eu C}\subseteq {\eu A}$ and ${\eu C}\subseteq {\eu B}$ are extensions of $\cL$-structures, then we will write ${\eu A} \equiv_{\eu C} {\eu B}$ if $({\eu A},{\eu C}) \equiv ({\eu B},{\eu C})$ in the language $\cL_{\eu C}$ augmented by constant names for the elements of ${\eu C}$. If for every two valued fields $(L,v), (F,v)\in{\bf C}$ and every common separably tame subfield $(K,v)$ of $(L,v)$ and $(F,v)$ such that $L|K$ and $F|K$ are separable, $vL/vK$ is torsion free and $Lv|Kv$ is separable, the side conditions $vL\equiv_{vK} vF$ and $Lv\equiv_{Kv} Fv$ imply that $(L,v)\equiv_{(K,v)} (F,v)$ in $\cL_{\VF}$, then we call {\bf C} {\bf separably relatively subcomplete}. If we can drop the requirements of $L|K$ and $F|K$ being separable, then we say {\bf C} is {\bf relatively subcomplete}. Analogously, we define the notion of {\bf separably relatively model complete}, which is basically relatively model complete in $\cL_{\VF}$ over separable extensions only. Note that if {\bf C} is a relatively subcomplete class of separably tame fields, then {\bf C} is relatively model complete in $\cL_{\VF}$ : the side conditions $vK\prec vL$ and $Kv\prec Lv$ imply that $vL/vK$ is torsion free and $Lv|Kv$ is separable (by Lemma~\ref{sica}) and that $vK\equiv_{vK} vL$ and $Kv\equiv_{Kv}Lv$. Hence, if {\bf C} is relatively subcomplete, then we obtain $(K, v)\equiv_{(K, v)} (L, v)$, i.e., $(K, v)\prec (L, v)$. But relative model completeness is weaker than relative subcompleteness, because $vL\equiv_{vK} vF$ does not imply that $vK\prec vL$, and $Lv\equiv_{Kv} Fv$ does not imply that $Kv\prec Lv$. Analogously, separable relative subcompleteness implies separable relative model completeness, but not vice versa.

The following lemma shows that the Separable Relative Embedding Property is a powerful property.
\begin{lem}		\label{REPimp}
Take an elementary class {\bf C} of separably tame fields which has the Separable Relative Embedding Property. Assume, in addition, that all fields in {\bf C} of positive characteristic have a fixed finite $p$-degree $e$. Then {\bf C} is separably relatively subcomplete and separably relatively model complete. The AKE$^\exists$ Principle is satisfied in $\cL_{\VF}$ by all separable extensions $(L|K, v)$ such that both $(K, v), (L, v)\in {\bf C}$. Moreover, if all fields in {\bf C}\linebreak are of fixed equal characteristic, then {\bf C} is relatively complete in $\cL_{\VF}$ (equivalently, in $\cL_\rQ$).

In the case when additionally all fields in {\bf C} are of a fixed characteristic, then {\bf C} is relatively subcomplete and relatively model complete in $\cL_{\rQ}$, and the AKE$^\exists$ Principle is satisfied in $\cL_\rQ$ by all extensions $(L|K, v)$ such that both $(K, v), (L, v)\in{\bf C}$.
\end{lem}
\begin{proof}
Let us first show that $(L|K, v)$ satisfies the AKE$^\exists$ Principle in $\cL_{\VF}$ whenever $(K, v), (L, v)\in{\bf C}$. So, assume that $vK\prec_\exists vL$ and $Kv\prec_\exists Lv$. We take an $\aleph_0$-saturated elementary extension $(\widetilde{L}, v)$ of $(L, v)$. Since ${\bf C}$ is assumed to be an elementary class, $(L, v)\in{\bf C}$ implies that $(\widetilde{L}, v)\in{\bf C}$. Because $\prec_\exists$ is transitive, we have $vK\prec_\exists v\widetilde{L}$ and $Kv\prec_\exists \widetilde{L}v$. Now take an $|\widetilde{L}|^+$-saturated elementary extension $(K^*, v^*)$ of $(K, v)$. Since ${\bf C}$ is elementary, we have $(K^*, v^*)\in{\bf C}$. By Proposition~\ref{ec}, there are embeddings $\rho: v\widetilde{L}\longrightarrow v^*K^*$ over $vK$ and $\sigma: \widetilde{L}v\longrightarrow K^*v^*$ over $Kv$. Moreover, $v\widetilde{L}/vK$ is torsion free, and $\widetilde{L}|L$ and $\widetilde{L}v|Kv$ are separable by Lemma~\ref{sica}. Since $L|K$ is also separable, we get that $\widetilde{L}|K$ is separable. By the Separable Relative Embedding Property, there is an embedding of $(\widetilde{L}, v)$ in $(K^*, v^*)$ over $K$, which shows that $(K, v)\prec_\exists(\widetilde{L}, v)$ in $\cL_{\VF}$. Hence, $(K, v)\prec_\exists(L, v)$ in $\cL_{\VF}$ as well. In the case when $K$ and $L$ are of characteristic zero, this automatically implies $(K, v)\prec_\exists(L, v)$ in $\cL_\rQ$. And in the case when $K$ and $L$ are of a fixed positive characteristic $p$, they have the same finite $p$-degree $e$ by assumption, and hence share a common $p$-basis. By Lemma~\ref{VF2Q}, we then have $(K, v)\prec_\exists(L, v)$ in $\cL_\rQ$.

Now assume that $(L, v), (F, v)\in{\bf C}$ with common separably tame subfield $(K, v)$ such that $L|K$ and $F|K$ are separable, $vL/vK$ is torsion free, $Lv|Kv$ is separable, $vL\equiv_{vK} vF$ and $Lv\equiv_{Kv} Fv$. We have to show that $(L, v)\equiv_{(K, v)} (F, v)$ in $\cL_{\VF}$. If $L$ and $F$ are both of characteristic zero, then the result follows from the corresponding result for tame fields \cite[Lemma 6.1]{FVK_TF}. So, assume without loss of generality that both $L$ and $F$ are of positive characteristic. We can also assume without loss of generality that $(L, v)$ and $(F, v)$ are $\aleph_0$-saturated.

The rest of the proof of this lemma is essentially the same as in \cite[Lemma 6.1]{FVK_TF}. First, we move to elementary extensions $(L_0, v)$ of $(L, v)$ and $(F_0, v)$ of $(F, v)$ such that $vL_0 = vF_0$ and $L_0v = F_0v$. Then we construct two elementary chains  $((L_i, v))_{i < \omega}$ and $((F_i, v))_{i < \omega}$ such that for each $i\ge 0$, the valued fields $(L_{i+1}, v)$ and $(F_{i+1}, v)$ are $\kappa_i^+$-saturated elementary extensions of $(L_i, v)$ and of $(F_i, v)$ respectively, where $\kappa_i = \max\{|L_i|, |F_i|\}$, and such that $vL_i = vF_i$ and $L_iv = F_iv$ for all $i\in\omega$. Since {\bf C} is elementary, all $(L_i, v)$ and $(F_i, v)$ are in {\bf C}. In particular, all of them are separably tame and have the same positive characteristic $p$ and the same finite $p$-degree $e$. Moreover, all $(L_i, v)$ and $(F_i, v)$ are trivially $\aleph_0$-saturated (since $\kappa\ge\aleph_0$). Taking $(L^*, v)$ and $(F^*, v)$ to be the unions over the elementary chains $((L_i, v))_{i < \omega}$ and $((F_i, v))_{i < \omega}$ respectively, we get $(L, v)\prec(L^*, v)$ and $(F, v)\prec(F^*, v)$. We then carry out a back-and-forth construction using the Separable Relative Embedding Property to show that $(L^*, v)$ and $(F^*, v)$ are isomorphic over $K$. For ease of presentation, we set $F_{-1} = L_{-1} := K$.

We start by embedding $(L_0, v)$ in $(F_1, v)$. The identity mappings are embeddings of $vL_0$ in $vF_1$ over $vK$ and of $L_0v$ in $F_1v$ over $Kv$, and we know that $vL_0/vK$ is torsion free and $L_0|K$ and $L_0v|Kv$ are separable. Since $(L_0, v)$ is $\aleph_0$-saturated, $(F_1, v)$ is $\kappa_0^+$-saturated with $\kappa_0\ge |L_0|$, and $(K, v)$ is separably tame, we can apply the Separable Relative Embedding Property to find an embedding $\iota_0$ of $(L_0, v)$ in $(F_1, v)$ over $K$ which respects the embeddings of the value group and the residue field. That is, we have that $v\iota_0L_0 = vF_0$ and $(\iota_0L_0)v = F_0v$.

The rest of the proof is very similar to the proof of \cite[Lemma 6.1]{FVK_TF}. The only additional thing to note is that at the $(i+1)^{st}$ step (for even $i$), while embedding $({\iota}_{i}^{-1}F_{i+1}, v)$ in $(L_{i+2}, v)$ over $L_{i}$, we can apply the Separable Relative Embedding Property because the extension ${\iota}_{i}^{-1}F_{i+1}|L_{i}$ is separable. This is because the isomorphic extension $F_{i+1}|\iota_iL_i$ is separable, which can be seen as follows: Let $\cB$ be a $p$-basis of $F_{i-1}$. Since $F_{i+1}|F_{i-1}$ is separable, it follows that $\cB$ is $p$-independent in $F_{i+1}$. Since $F_{i-1}\subseteq\iota_iL_i\subseteq F_{i+1}$, $\cB$ is also $p$-independent in $\iota_iL_i$. Since $\iota_iL_i$ has the same $p$-degree $e$ as $F_{i-1}$, $\cB$ is in fact a $p$-basis of $\iota_iL_i$. And thus, by Theorem~\ref{pindep}, the extension $F_{i+1}|\iota_iL_i$ is separable. The same applies to the embedding of $({\iota'}_{i+1}^{-1}L_{i+2}, v)$ in $(F_{i+3}, v)$ over $F_{i+1}$ at the $(i+2)^{nd}$ step, where the embedding $\iota'_{i+1} : (F_{i+1}, v)\to (L_{i+2}, v)$ over $K$ is as defined in \cite[Lemma 6.1]{FVK_TF}.


Thus, eventually, we obtain an isomorphism from $(L^*, v)$ onto $(F^*, v)$ over $K$, which shows that $(L^*, v)\equiv_{(K, v)} (F^*, v)$ in $\cL_{\VF}$. Since $(L, v)\prec(L^*, v)$ and $(F, v)\prec(F^*, v)$, this implies that $(L, v)\equiv_{(K, v)}(F, v)$ in $\cL_{\VF}$, as required. We have proved that {\bf C} is separably relatively subcomplete, and we know already from the beginning of this section that this implies that {\bf C} is separably relatively model complete. In the case when all fields in {\bf C} have the same characteristic, we can use the language $\cL_\rQ$. Since in the language $\cL_\rQ$ every structure is automatically separable over any substructure, it follows immediately that {\bf C} is relatively subcomplete, and hence relatively model complete, in the language $\cL_\rQ$.

Finally, assume in addition that all fields in {\bf C} are of fixed equal characteristic. We wish to show that {\bf C} is relatively complete in $\cL_{\VF}$. So, take $(L,v), (F,v)\in {\bf C}$ such that $vL\equiv vF$ and $Lv\equiv Fv$. Fixed characteristic means that $L$ and $F$ have a common prime field $K$. Since we are in the equal characteristic case, the restrictions of their valuations to $K$ is trivial. Hence, $vK = 0$ and consequently, $vL/vK$ is torsion free and $vL\equiv vF$ implies that $vL\equiv_{vK} vF$. Further, $K=Kv$ is also the prime field of $Lv$ and $ Fv$, so $Lv\equiv Fv$ implies that $Lv\equiv_{Kv} Fv$. Since a prime field is always perfect, we also have that $L|K$, $F|K$ and $Lv|Kv$ are all separable. As a trivially valued field, $(K,v)$ is separably tame. From what we have already proved, we obtain that $(L,v)\equiv_{(K,v)} (F,v)$ in $\cL_{\VF}$, which implies that $(L,v)\equiv (F,v)$ in $\cL_{\VF}$. Again, since all fields in {\bf C} have the same characteristic, we can use the language $\cL_\rQ$, and since the predicates $Q_m$, for $m\ge 1$, are definable in $\cL_{\VF}$, it follows that $(L,v)\equiv (F,v)$ in $\cL_{\rQ}$ as well.
\end{proof}

Our goal now is to give a criterion for an elementary class {\bf C} of valued fields to have the Separable Relative Embedding Property. Since Embedding Lemma II (Lemma~\ref{EII}) covers the case of extensions without transcendence defect, we are left to deal with the case of extensions $(L|K, v)$ with transcendence defect. Loosely speaking, these contain an immediate part. We will show that this part can be treated separately, that is, we can find an intermediate field $(L', v)\in {\bf C}$ such that $(L|L', v)$ is immediate and $(L'|K, v)$ has no transcendence defect. The immediate part is then handled by the following theorem. For a proof of this theorem, see \cite[Theorem 1.7(b)]{FVK_TF}.
\begin{thm}		\label{stake}
Every separable extension $(L|K, v)$ of a separably tame field satisfies the AKE$^\exists$ Principle in $\cL_{\VF}$.
\end{thm}

Now we are able to give the announced criterion.
\begin{lem}		\label{critREP}
Let {\bf C} be an elementary class of valued fields which satisfies
\begin{axiom}
\ax{(CSTF)} every field in {\bf C} is separably tame,
\ax{(CRAC)} if $(L, v)\in {\bf C}$ and $K$ is relatively algebraically closed in $L$ such that $Lv|Kv$ is algebraic and $vL/vK$ is a torsion group, then $(K, v)\in {\bf C}$ with $Lv= Kv$ and $vL=vK$.
\end{axiom}
Then {\bf C} has the Separable Relative Embedding Property.
\end{lem}
\begin{proof}
Assume that the elementary class {\bf C} satisfies (CSTF) and (CRAC). Take $(L, v)$, $(K^*, v^*)\in {\bf C}$ with $(L, v)$ being $\aleph_0$-saturated, $(K^*, v^*)$ being $|L|^+$-saturated, a separably tame subfield $(K, v)$ of $(L, v)$ and $(K^*, v^*)$ such that $L|K$ is separable, $vL/vK$ is torsion free and $Lv|Kv$ is separable, and embeddings $\rho:\; vL \longrightarrow v^*K^*$ over $vK$ and $\sigma:\>Lv\longrightarrow K^*v^*$ over $Kv$. We have to show that there exists an embedding $\iota:\>(L, v) \longrightarrow (K^*, v^*)$ over $K$ which respects $\rho$ and $\sigma$.

If $K$ (and hence, $L$ and $K^*$) are of characteristic zero, then the result follows from the corresponding result for tame fields \cite[Lemma 6.4]{FVK_TF}. So, assume without loss of generality that they are of positive characteristic $p$. Moreover, if $(L, v)$ is trivially valued, then  the required embedding $\iota$ can be constructed by lifting $\sigma$. So, assume without loss of generality that $(L, v)$ (and consequently $(K^*, v^*)$) is nontrivially valued.

Now take a set $\cT = \{x_i\,,\,y_j\mid i\in I\,,\,j\in J\}$ such that the values $\{vx_i \mid i\in I\}$ form a maximal set of values in $vL$ rationally independent over $vK$, the residues $\{y_jv \mid j\in J\}$ form a transcendence basis of $Lv|Kv$, and the elements $\{x_i^{1/p^n}, y_j^{1/p^n} \mid i\in I, j\in J, n\in\bN\}$ are all in $L$. This is possible to do because $(L, v)$ is $\aleph_0$-saturated by assumption, and  $Lv$ is perfect and $vL$ is $p$-divisible by Lemma~\ref{septame}. For gory details on how saturation is used in this context, see \cite[Lemme 3.5]{FD_thesis}.

Now $vL/vK(\cT)$ is a torsion group and $Lv|K(\cT)v$ is algebraic. Let $K'$ be the relative algebraic closure of $K(\cT)$ within $L$. It follows that also $vL/vK'$ is a torsion group and $Lv|K'v$ is algebraic. Hence, by condition (CRAC), we have that $(K', v)\in {\bf C}$ with $Lv= K'v$ and $vL=vK'$, which shows that the extension $L|K'$ is immediate. Moreover, $\cB\cup\cT$ is a $p$-basis of $K(\cT)$, where $\cB$ is a $p$-basis of $K$. By the way we have chosen $\cT$, it follows that $\cT\subseteq K'^{p^n}$ for every $n\in\bN$. Since $L|K$ is separable, it follows that $\cB$ is $p$-independent in $L$, and hence in $K'$. Thus, $\cB$ is also a $p$-basis of $K'$. In particular, $L|K'$ is separable.

On the other hand, $\cT$ is a standard valuation transcendence basis of $(K'|K, v)$ by construction. Hence, according to \cite[Corollary 2.4]{FVK_TF}, this extension has no transcendence defect. Also, $K'|K$ is separable being a subextension of the separable extension $L|K$. Now $(K^*, v^*)$ is separably tame, and hence by Lemma~\ref{phtd}, $({K^*}^{1/p^\infty}, v^*)$ is tame, in particular henselian, and is contained in $({K^*}^c, v^*)$. Similarly, $(K^{1/p^\infty}, v)$ is tame, in particular defectless. By what we have seen in the proof of Theorem~\ref{sewtd}, for every finitely generated subextension $F|K$ of $K'|K$, we have that the extension $(FK^{1/p^\infty}|K^{1/p^\infty}, v)$ is a valued function field without transcendence defect satisfying the conditions of $vFK^{1/p^\infty}/vK^{1/p^\infty}$ being torsion free and $FK^{1/p^\infty}v|K^{1/p^\infty}v$ being separable, and also having corresponding embeddings of the value group and residue field of $(FK^{1/p^\infty}, v)$ in those of $({K^*}^{1/p^\infty}, v^*)$ over those of $(K^{1/p^\infty}, v)$ extending the given $\rho$ and $\sigma$, respectively. Since $(K^{1/p^\infty}, v)$ is defectless, it follows from Theorem~\ref{hrwtd} that $(FK^{1/p^\infty}|K^{1/p^\infty}, v)$ is strongly inertially generated. Hence, from Lemma~\ref{ael} we obtain an embedding $j:(FK^{1/p^\infty}, v)\longrightarrow({K^*}^{1/p^\infty}, v^*)$ (and hence in $({K^*}^c, v^*)$) over $K^{1/p^\infty}$ that respects $\rho$ and $\sigma$. By taking the restriction of $j$ to $(F, v)$, we obtain an embedding $j: (F, v)\longrightarrow({K^*}^c, v^*)$ over $K$ that respects $\rho$ and $\sigma$. Since this is true for every finitely generated subextension $F|K$ of $K'|K$, and since $(K^*, v^*)$ is henselian by condition (CSTF) and is also an $|L|^+$-, and hence $|K'|^+$-, saturated extension of $(K, v)$, Embedding Lemma II (Lemma~\ref{EII}) gives an embedding of $(K', v)$ in $(K^*, v^*)$ over $K$ that respects $\rho$ and $\sigma$. Now we have to look for an extension of this embedding to $(L, v)$.

We identify $K'$ with its image in $K^*$. Since $L|K'$ is separable and $(K', v)$ is separably tame, the extension $(L|K', v)$ satisfies the AKE$^\exists$ Principle in $\cL_{\VF}$ by Theorem~\ref{stake}. Since $(K^*, v^*)$ is $|L|^+$-saturated, it follows by Proposition~\ref{ec} that $(L, v)$ embeds in $(K^*, v^*)$ over $K'$. Finally, since $(L|K', v)$ is immediate, such an extension automatically respects $\rho$ and $\sigma$. This completes our proof.
\end{proof}

\section{The Model Theory of Separably Tame Fields}
We first show that the property of being a separably tame field of a fixed residue characteristic is elementary. In the residue characteristic zero case, a separably tame field is trivially a tame field, which in turn is equivalent to being a henselian field, which is axiomatized by the axiom scheme (HENS$_n)_{n < \omega}$ mentioned below. Now assume that the residue characteristic is fixed to be a prime $p$. As mentioned in the Introduction, every trivially valued field $(K, v)$ is separably tame. On the other hand, by Lemma~\ref{septame}, a nontrivially valued field of positive residue characteristic is separably tame if and only if it is a separable-algebraically maximal field having $p$-divisible value group and perfect residue field. A valued field $(K,v)$ has $p$-divisible value group if and only if it satisfies the following elementary axiom:
\begin{axiom}
\ax{(VGD$_p$)} $\forall x\,\exists y\>(vxy^p = 0)\,.$
\end{axiom}
Furthermore, $(K,v)$ has perfect residue field if and only if it satisfies:
\begin{axiom}
\ax{(RFD$_p$)} $\forall x\,\exists y\>\Big((vx=0)\>\longrightarrow\>(v(xy^p - 1) > 0)\Big)\,.$
\end{axiom}
Finally, in view of Theorem~\ref{exmaxvfth}, the property of being separable-algebraically maximal is axiomatized by the axiom schemes (HENS$_n)_{n \in\omega}$ and (SMAXP$_n)_{n \in\omega}$:
\begin{axiom}
\ax{(HENS$_n$)} $\forall x\forall y\Big((vy\geq 0\>\wedge\>\bigwedge_{1\leq i\leq n} vx_i\geq 0 \>\wedge\> v(y^n+x_1y^{n-1}+\cdots +x_{n-1}y +x_n)>0\\
\hspace*{5.1cm}\wedge\>v(ny^{n-1}+(n-1)x_1y^{n-2}+\cdots+x_{n-1})=0)\\
\longrightarrow\> \exists z\;(v(y-z)>0\,\wedge\,z^n+x_1z^{n-1}+\cdots +x_{n-1}z +x_n=0)\Big)$\\
\ax{(SMAXP$_n$)} $\forall x_1\cdots \forall x_n\Big(\bigvee_{\begin{tabular}{l}$1\le i\le n$\\$p\not|\;i$\end{tabular}}(x_i\not=0) \longrightarrow \\
\exists y\forall z\>\Big(v(y^n+x_1y^{n-1}+\cdots+x_{n-1}y+x_n) \geq v(z^n+x_1z^{n-1}+\cdots+x_{n-1}z+x_n)\Big)\Big)\,.$\\
\end{axiom}
We summarize: The {\bf theory of trivially valued separably tame fields} is just the theory of fields. The {\bf theory of separably tame fields of residue characteristic $0$} is just the theory of henselian fields of residue characteristic $0$. If $p$ is a prime, then the {\bf theory of nontrivially valued separably tame fields of residue characteristic $p$} is the theory of valued fields together with axioms (VGD$_p$), (RFD$_p$), (HENS$_n)_{n \in\omega}$ and (SMAXP$_n)_{n \in\omega}$. Now we also see how to axiomatize the theory of all separably tame fields. Indeed, for residue characteristic 0 or for trivial valuation, there are no conditions on the value group and the residue field. For residue characteristic $p > 0$ and nontrivial valuation, we have to require (VGD$_p$) and (RFD$_p$). We can do this by the axiom scheme (TAD$_p)_{p\in\omega}$:
\begin{axiom}
\ax{(TAD$_p$)} $\Big(v(\underbrace{1+\cdots+1}_{p \ \rm times}) > 0 \;\wedge\; \exists x\, (x\not=0\wedge v(x) > 0)\Big) \> \longrightarrow\>(\mbox{\rm (VGD$_p$)}\,\wedge\, \mbox{\rm (RFD$_p$)}).$
\end{axiom}
So, the {\bf theory of separably tame fields} is the theory of valued fields together with axioms (TAD$_p)_{p\in\omega}$, (HENS$_n)_{n \in\omega}$ and (SMAXP$_n)_{n \in\omega}\,$.


Now we will develop the model theory of separably tame fields. Let {\bf STVF} be the elementary class of all separably tame fields. Let {\bf STVF$_p$} be the elementary subclass of {\bf STVF} such that all fields are of the same characteristic $p$, where $p = 0$ or $p$ is a prime. For prime $p$, we can express this by the axiom (FEC$_p)$:
\begin{axiom}
\ax{(FEC$_p$)} $\underbrace{1+\cdots+1}_{p \ \rm times} = 0\,$.
\end{axiom}
And for $p = 0$, we can express this by the axiom scheme $(\neg$FEC$_p)_{p\in\omega}$.

Finally, let {\bf STVF$_{0, 0}$} be the elementary subclass of {\bf STVF$_0$} such that all fields are of equicharacteristic zero, which can be stated by the axiom scheme (EC0$_p)_{p\in\omega}$:
\begin{axiom}
\ax{(EC0$_p$)} $v(\underbrace{1+\cdots+1}_{p \ \rm times}) = 0\,$;
\end{axiom}
and for $p > 0$, let {\bf STVF$_{p,\, e}$} be the elementary subclass of {\bf STVF$_p$} such that all fields have a fixed finite $p$-degree $e\in\bN$, which can be expressed in the language $\cL_\rQ$ by the axiom (FPD$_{p,\, e})$:
\begin{axiom}
\ax{(FPD$_{p,\, e}$)} $\exists x_1 \cdots\exists x_e\,Q_e(x_1,\ldots, x_e)\;\wedge\;\forall x_1\cdots\forall x_{e+1}\neg Q_{e+1}(x_1, \ldots, x_{e+1})\,$.
\end{axiom}
By Lemma~\ref{Xsrac} and Theorem~\ref{stt3}, {\bf STVF} satisfies conditions (CRAC) and (CIMM). Hence, we can infer from Lemma~\ref{critREP} and Lemma~\ref{REPimp}:

\begin{thm}		\label{septamemod1}
The elementary class {\bf STVF} has the Separable Relative Embedding Property. Thus, the elementary classes {\bf STVF$_0$} and {\bf STVF$_{p,\, e}$} are separably relatively subcomplete and separably relatively model complete in $\cL_{\VF}$, and relatively subcomplete and relatively model complete in $\cL_\rQ$, and are AKE$^\exists$-classes in $\cL_\rQ$. Also, the elementary classes {\bf STVF$_{0, 0}$} and {\bf STVF$_{p,\, e}$} are relatively complete in $\cL_{\VF}$ (equivalently, in $\cL_\rQ$).
\end{thm}

Lemma~\ref{REPimp} does not give the full information about the AKE$^\exists$  Principle because it requires that not only $(K,v)$, but also $(L,v)$ is a member of the class {\bf C}. For an improved result, see Theorem~\ref{stake}.

Now let {\bf C} be an elementary class of valued fields. We define
\[v{\bf C}\,:=\,\{vK\mid (K,v)\in {\bf C}\}\;\mbox{\ \ and\ \ }\;{\bf C}v\,:=\,\{Kv\mid (K,v)\in {\bf C}\}\;.\]
If both $v{\bf C}$ and ${\bf C}v$ are model complete elementary classes, then the side conditions $vK\prec vL$ and $Kv\prec Lv$ will hold for every two members $(K, v)\subseteq (L, v)$ of {\bf C}. Similarly, if $v{\bf C}$ and ${\bf C}v$ are complete elementary classes, then the side conditions $vK\equiv vL$ and $Kv\equiv Lv$ will hold for all $(K, v), (L, v)\in {\bf C}$. So, we obtain from Theorem~\ref{septamemod1} and Theorem~\ref{dec} the following result.
\begin{thm}		\label{cormodtf}
If {\bf C} is an elementary class consisting of separably tame fields of a fixed positive characteristic $p$ and a fixed finite $p$-degree $e$, and if $v{\bf C}$ and ${\bf C}v$ are elementary model complete classes, then {\bf C} is model complete in $\cL_{\rQ}$. Analogously, if $v{\bf C}$ and ${\bf C}v$ are elementary complete classes, then {\bf C} is complete in $\cL_{\VF}$. And if $v{\bf C}$ and ${\bf C}v$ admit recursive elementary axiomatizations, then ${\bf C}$ is decidable in $\cL_{\VF}$.
\end{thm}
Note that the converses are true by virtue of Corollary~\ref{vgrfequiv}, provided that $v{\bf C}$ and ${\bf C}v$ are elementary classes. As an immediate consequence we get the following. Fix a prime $p$. Let $G$ be a nontrivial $p$-divisible ordered abelian group, and $k$ be a perfect field of characteristic $p$. Then we have
\begin{thm}
Let {\bf STVF$^{G,\, k}_{p,\, e}$} be the elementary subclass of {\bf STVF$_{p,\, e}$} such that all valued fields have value group and residue field elementarily equivalent to $G$ and $k$, respectively. Then {\bf STVF$^{G,\, k}_{e,\, p}$} is complete in $\cL_{\VF}$. If $\mathrm{Th}(G)$ and $\mathrm{Th}(k)$ are model complete in their respective languages, then {\bf STVF$^{G,\, k}_{e,\, p}$} is also model complete in $\cL_\rQ$. And if $\mathrm{Th}(G)$ and $\mathrm{Th}(k)$ admit recursive elementary axiomatizations in their respective languages, then {\bf STVF$^{G,\, k}_{e,\, p}$} is decidable in $\cL_{\VF}$.
\end{thm}
As a final example, we consider the theory of separably tame fields of fixed positive characteristic and fixed finite $p$-degree with nontrivial divisible or $p$-divisible value groups and fixed finite residue field.

\begin{thm}		\label{sptame}
a) \ Every elementary class {\bf C} of separably tame fields of fixed positive characteristic $p$ and fixed finite $p$-degree $e$ with nontrivial divisible value group and fixed residue field $\bF_q$ (where $q = p^n$ for some $n\in\bN$) is model complete in the language $\cL_\rQ$, and complete and decidable in the language $\cL_{\VF}$.\newline
b) \ If ``divisible value group'' is replaced by ``value group elementarily equivalent to $\frac{1}{p^{\infty}}\bZ$'', then {\bf C}
remains elementary, complete and decidable in the language $\cL_{\VF}$.
\end{thm}
\begin{proof}
a) \ It is well-known that the elementary theory of nontrivial divisible ordered abelian groups is model complete, complete and decidable \cite{DM}. The same holds trivially for the elementary theory of the finite field $\bF_q$ which has $\bF_q$ as the only model up to isomorphism. Hence, model completeness, completeness and decidability follow readily from Theorem~\ref{cormodtf}. \newline
b) \ The elementary theory of $\frac{1}{p^{\infty}}\bZ$ is clearly complete, and it is decidable (and {\bf C} is still elementary) because it can be axiomatized by a recursive set of elementary axioms. Now the proof proceeds as in part a), except that we replace $\bQ$ by $\frac{1}{p^{\infty}}\bZ$ and note that the latter admits an elementary embedding in every elementarily equivalent ordered abelian group \cite{AREZ}.
\end{proof}

Note that in the case of b), model completeness can be reinstated by adjoining a constant symbol to the language $\cL_\rQ$ and by adding axioms that state that the value of the element named by this symbol is divisible by no prime but $p$.

\section{More Applications}		\label{app}
Now we look at two more applications of Lemma~\ref{critREP} and Lemma~\ref{REPimp}. The first application is for the class of separable-algebraically closed valued fields and the second application is for the class of separable-algebraically maximal Kaplansky fields, both of which were considered by Delon \cite{FD_thesis}. Our approach gives an alternate proof to these well-known results.

\subsection{Separably Closed Valued Fields}
Let {\bf SCVF} be the elementary class of all separable-algebraically closed valued fields. Let {\bf SCVF$_p$} be the elementary subclass of {\bf SCVF} in which all fields are of fixed characteristic $p$, where $p = 0$ or $p$ is a prime. Let {\bf SCVF$_{0, 0}$} be the elementary subclass of {\bf SCVF$_0$} in which all fields are of equicharacteristic 0; and for $p > 0$, let {\bf SCVF$_{p,\, e}$} be the elementary subclass of {\bf SCVF$_p$} such that all fields have a fixed finite $p$-degree $e$. Finally, let {\bf SCVF$^{v\not=0}_{p,\, e}$} be the elementary subclass of {\bf SCVF$_{p,\, e}$} in which all valued fields are nontrivially valued. Then we have

\begin{thm}		\label{sepclosed}
{\bf SCVF$^{v\not=0}_{p,\, e}$} is model complete in $\cL_\rQ$, and complete and decidable in $\cL_{\VF}$.
\end{thm}
\begin{proof}
We first show that the elementary class {\bf SCVF} has the Separable Relative Embedding Property. Separable-algebraically closed valued fields are trivially separable-algebraically maximal. By Lemma~\ref{sep_perf_hull} and Lemma~\ref{septame}, every separable-algebraically closed nontrivially valued field is separably tame. On the other hand, every separable-algebraically closed trivially valued field is trivially a separably tame field. Thus, {\bf SCVF} satisfies (CSTF). Now suppose $(L, v)$ is a separable-algebraically closed valued field and $K$ is relatively algebraically closed in $L$ such that $Lv|Kv$ is algebraic and $vL/vK$ is a torsion group. Since $(L, v)$ is separably tame, it follows by Lemma~\ref{Xsrac} that $(K, v)$ is separably tame with $Kv = Lv$ and $vK = vL$. Moreover, since $L$ is separable-algebraically closed, the separable-algebraic closure $K^{\sep}$ of $K$ is contained in $L$. Since $K$ is relatively algebraically closed in $L$, it therefore follows that $K = K^{\sep}$. In particular, $(K, v)$ is a separable-algebraically closed valued field with $Kv = Lv$ and $vK = vL$. Thus, {\bf SCVF} satisfies (CRAC). And hence, by Lemma~\ref{critREP}, the elementary class {\bf SCVF} has the Separable Relative Embedding Property.

Consequently, by Lemma~\ref{REPimp}, the elementary classes {\bf SCVF$_0$} and {\bf SCVF$_{p,\, e}$} are separably relatively subcomplete and separably relatively model complete in $\cL_{\VF}$, and relatively subcomplete and relatively model complete in $\cL_\rQ$, and are AKE$^\exists$-classes in $\cL_\rQ$. Also, the elementary classes {\bf SCVF$_{0, 0}$} and {\bf SCVF$_{p,\, e}$} are relatively complete in $\cL_{\VF}$ (equivalently, in $\cL_\rQ$). Our result then immediately follows from the well-known fact that the classes of nontrivial divisible ordered abelian groups and algebraically closed fields are model complete, complete and decidable in their respective languages \cite{DM}.
\end{proof}

\subsection{Separable-algebraically Maximal Kaplansky Fields}
Let {\bf SMKF} be the elementary class of all separable-algebraically maximal Kaplansky fields. Let {\bf SMKF$_p$} be the elementary subclass of {\bf SMKF} in which all fields are of fixed characteristic $p$, where $p = 0$ or $p$ is a prime. Finally, let {\bf SMKF$_{0, 0}$} be the elementary subclass of {\bf SMKF$_0$} such that all fields are of equicharacteristic 0; and for $p > 0$, let {\bf SMKF$_{p,\, e}$} be the elementary subclass of {\bf SMKF$_p$} such that all fields have a fixed finite $p$-degree $e$. Then we have

\begin{thm}		\label{sepclosedmod1}
The elementary class {\bf SMKF} has the Separable Relative Embedding Property. Thus, the elementary classes {\bf SMKF$_0$} and {\bf SMKF$_{p,\, e}$} are separably relatively subcomplete and separably relatively model complete in $\cL_{\VF}$, and relatively subcomplete and relatively model complete in $\cL_\rQ$, and are AKE$^\exists$-classes in $\cL_\rQ$. The elementary classes {\bf SMKF$_{0, 0}$} and {\bf SMKF$_{p,\, e}$} are relatively complete in $\cL_{\VF}$ (equivalently, in $\cL_\rQ$).
\end{thm}
\begin{proof}
By \cite[Corollary 3.11(a)]{FVK_TF}, all separable-algebraically maximal Kaplansky fields are separably tame. Thus, {\bf SMKF} satisfies (CSTF). Now suppose $(L, v)$ is a separable-algebraically maximal Kaplansky field and $K$ is relatively algebraically closed in $L$ such that $Lv|Kv$ is algebraic and $vL/vK$ is a torsion group. Again, since $(L, v)$ is separably tame, it follows by Lemma~\ref{Xsrac} that $(K, v)$ is separably tame with $Kv = Lv$ and $vK = vL$. In particular, $(K, v)$ is separable-algebraically maximal. Moreover, since $(L, v)$ is a Kaplansky field and $(L|K, v)$ is immediate, $(K, v)$ also satisfies the Kaplansky conditions. Thus, $(K, v)$ is a separable-algebraically maximal Kaplansky field with $Kv = Lv$ and $vK = vL$. Hence, {\bf SMKF} satisfies (CRAC). By Lemma~\ref{critREP}, the elementary class {\bf SMKF} has the Separable Relative Embedding Property. The rest follows by Lemma~\ref{REPimp}.
\end{proof}

Fix a prime $p$. Let $G$ be a nontrivial $p$-divisible ordered abelian group, and $k$ be a field of characteristic $p$ satisfying: for all elements $a_0, \ldots, a_{n-1}, b\in k$, the equation
$$x^{p^n} + a_{n-1}x^{p^{n-1}} + \cdots + a_1x^p + a_0 x + b = 0$$
has a solution in $k$. Such a field is called {\em $p$-closed}. By a similar proof as before, we have
\begin{thm}
Let {\bf SMKF$^{G,\, k}_{p,\, e}$} be the elementary subclass of {\bf SMKF$_{p,\, e}$} such that all valued fields have value group and residue field elementarily equivalent to $G$ and $k$ respectively. Then {\bf SMKF$^{G,\, k}_{p,\, e}$} is complete in $\cL_{\VF}$. If $\mathrm{Th}(G)$ and $\mathrm{Th}(k)$ are model complete in their respective languages, then {\bf SMKF$^{G,\, k}_{p,\, e}$} is also model complete in $\cL_\rQ$. And if $\mathrm{Th}(G)$ and $\mathrm{Th}(k)$ admit recursive elementary axiomatizations in their respective languages, then {\bf SMKF$^{G,\, k}_{p,\, e}$} is decidable in $\cL_{\VF}$.
\end{thm}

\bigskip

\bibliographystyle{amsplain}
\bibliography{references}

\end{document}